\documentclass[10 pt, onecolumn, oneside, a4paper, reqno]{amsart}
\usepackage{amssymb}
\usepackage{mathrsfs}
\usepackage{mathtools}
\usepackage{amsmath}
\usepackage{tikz}
\usepackage{pifont}
\usepackage{color}
\usepackage[all]{xy}
\usepackage{type1cm}
\usepackage{hyperref}
\usepackage{color}
\hypersetup{colorlinks=true,
     breaklinks=true,
     linkcolor=blue,
      pageanchor=true}
\newtheorem{thm}{Theorem}[section]

\newtheorem{cor}[thm]{Corollary}
\newtheorem{lem}[thm]{Lemma}
\newtheorem{prop}[thm]{Proposition}

\theoremstyle{definition}
\newtheorem{defn}[thm]{Definition}
\newtheorem{exam}[thm]{Example}
\newtheorem{rem}[thm]{Remark}
\newcommand{\al}{\alpha}

\newcommand{\lra}{\longrightarrow}
\newcommand{\ga}{\gamma}
\newcommand{\Ga}{\Gamma}
\newcommand{\ra}{\rightarrow}

 \newcommand{\A}{\mathcal A}

\newcommand{\C}{\mathcal C}
\newcommand{\D}{\mathcal D}
\newcommand{\E}{\mathcal E}
\newcommand{\I}{\mathcal I}
\newcommand{\J}{\mathcal J}

\renewcommand{\P}{\mathcal P}

\newcommand{\R}{\mathcal R}
\newcommand{\HT}{\mathcal T}

\newcommand{\X}{\mathcal X}
\newcommand{\Y}{\mathcal Y}

\renewcommand{\dim}{\mathsf{dim}}

\DeclareMathOperator*{\add}{\mathsf{add}}

\DeclareMathOperator{\CoGh}{\mathsf{CoGh}}

\DeclareMathOperator{\End}{\mathsf{End}}

\DeclareMathOperator{\Gh}{\mathsf{Gh}}
\DeclareMathOperator{\Hom}{\mathsf{Hom}}
 \DeclareMathOperator*{\id}{\mathsf{id}}

  \DeclareMathOperator*{\ind}{\mathsf{ind}}
 \DeclareMathOperator*{\Irr}{\mathsf{Irr}}

\DeclareMathOperator*{\Mor}{\mathsf{Mor}}
\DeclareMathOperator*{\Ob}{\mathsf{Ob}}

 \DeclareMathOperator*{\smod}{\!-\mathsf{mod}}
\title[Ideal mutations]{Ideal mutations in triangulated categories and generalized Auslander-Reiten theory$^\text{~\ding{73}}$}

\author[]{Yaohua Zhang* and Bin Zhu}
\thanks{* Corresponding author.}

\thanks{$^\text{\ding{73}}$ Bin Zhu is supported by the National Natural Science Foundation of China (Grant Nos. 12371034  and  12031007).}

\thanks{Yaohua Zhang. Hubei Key Laboratory of Applied Mathematics, Faculty of Mathematics and Statistics, Hubei University, Wuhan, 430062, People's Republic of China {\em E-mail:}2160501008@cnu.edu.cn}

\thanks{Bin Zhu. Department of  Mathematical Sciences, Tsinghua University, Beijing 100048, People's Republic of China.  {\em E-mail:} zhu-b@mail.tsinghua.edu.cn}



\begin{document}
\begin{abstract}
We introduce the notion of ideal mutations in a triangulated category, which generalizes the version of Iyama and Yoshino \cite{iyama2008mutation} by replacing approximations by objects of a subcategory with approximations by morphisms of an ideal. As applications, for a $\Hom$-finite Krull-Schmidt triangulated category $\HT$ over an algebraically closed field $K$. (1) We generalize a theorem of Jorgensen \cite[Theorem 3.3]{jorgensen2010quotients} to a more general setting;  (2) We provide a method to detect whether $\HT$ has Auslander-Reiten triangles or not by checking the necessary and sufficient conditions on its Jacobson radical $\J$:  (i) $\J$ is functorially finite, (ii) $\Gh_\J=\CoGh_\J$, and (iii) $\Gh_\J$-source maps coincide with  $\Gh_\J$-sink maps; (3) We generalize the classical Auslander-Reiten theory by using ideal mutations.
\end{abstract}
\maketitle
\tableofcontents
\section{Introduction}
 Auslander-Reiten theory is a fundamental tool for the representation theory of Artin algebras \cite{auslander1995representation}. This theory has much influence on algebraic geometry and algebraic topology \cite{auslander1986almost, jorgensen2008calabi}. An analog theory for triangulated categories was introduced by Happel \cite{happel1988}. In this theory,  the concept Auslander-Reiten sequence, introduced firstly by Auslander and Reiten \cite{auslander1975representation}, plays a crucial role.

The idea of mutations in triangulated categories was introduced by Iyama and Yoshino \cite{iyama2008mutation} to classify rigid Cohen-Macaulay modules.  The method to make mutations is to do approximations by a subcategory. This kind of approximations is a special case of ideal approximations which was first introduced by Fu, Guil Asensio, Herzog and Torrecillas \cite{fu2013ideal}. In the ideal approximation theory, the role of the objects and subcategories in classical
approximation theory is replaced by morphisms and ideals of the category, and basic results, such as Salce's Lemma, Christehsen's Lemma and Wakamatsu's Lemma, have corresponding versions in this more general setting \cite{fu2016powers}.
There are many interesting examples of non-object ideal approximations, such as the phantom ideal approximations introduced by Herzog \cite{herzog2007phantom} and the $\P$-null approximations concerning a class $\P$ of objects by Christensen \cite{christensen1998ideals}. Indeed, the well-known source and sink maps are left and right approximations of the Jacobson radical ideal, respectively.

In the present paper, we will introduce the notion of ideal mutations of triangulated categories via approximations by morphisms from an ideal. It is a generalization of the version of Iyama and Yoshino \cite{iyama2008mutation} (\cite{liu2013triangulated} or \cite{zhou2018triangulated}). In this new framework, some known results can be generalized, including a Jorgensen's theorem \cite[Theorem 3.3]{jorgensen2010quotients} (or \cite[Corollary 4.4]{zhou2018triangulated}). As its applications and advantages over the mutations of Iyama and Yoshino { we can mention} an equivalent characterization of a triangulated category having Auslander-Reiten triangles by checking its Jacobson radical. Then, we will use the notion of ideal mutations to generalize the classical Auslander-Reiten theory of triangulated categories, this helps us have a better understanding of the Auslander-Reiten theory.

To state our results precisely, let us briefly introduce some terminology.
Let $K$ be an algebraically closed field and $\HT$  a $\Hom$-finite  Krull-Schmidt triangulated $K$-category. Let $\I$ be an ideal of $\HT$. By $\I[1]$ we denote the shift ideal $\{f[1]~|~f\in\I\}$ and by $\Ob(\I)$  we denote the full subcategory of $\HT$ with objects $X$ satisfying $\id_X\in\I$. We define $\Gh_\I=\{f\in\Mor \HT~|~fi=0, \forall i\in\I\}$ and $\CoGh_\I=\{f\in\Mor \HT~|~if=0, \forall i\in\I\}$, they are both ideals of $\HT$. A morphism $f:X\to Y$ is a {\em right $\I$-approximation} of $Y$ if $f\in\I$ and each $i:M\to Y$ in $\I$ factors through $f$. A {\em left $\I$-approximation} is defined dually. An ideal $\I$ is {\em functorially finite} if each object has both right and left $\I$-approximations. $\I$ is called an {\em Auslander-Reiten ideal} if $(\HT, \HT)$ is  an {\em $\I$-mutation pair} (see Definition~\ref{defn: I mutation pair}), that is, for $X\in\HT$ there exist triangles
$$X\stackrel{f}\lra M\stackrel{g}\lra Y\stackrel{h}\lra X[1]~\text{and}~Z\stackrel{u}\lra N\stackrel{v}\lra  X\stackrel{w}\lra X[1]$$
such that $f, u$ are left $\I$-approximations and $g, v$ are right $\I$-approximations. Our first main result is to give a characterization of Auslander-Reiten ideals.

\begin{thm}[Theorem~\ref{thm:char of AR ideal}]
  Let $\I$ be a functorially finite ideal in $\HT$. Then $\I$ is an Auslander-Reiten ideal if and only if
  \begin{enumerate}
    \item $\Gh_\I=\CoGh_{\I[1]}$, and moreover
    \item  for $X\in\ind\HT\setminus\Ob\I$, each $\Gh_\I$-source map of $X$ is a  $\Gh_\I$-sink map and each $\Gh_\I$-sink map of $X[1]$ is a $\Gh_\I$-source map.
  \end{enumerate}
\end{thm}

When we restrict on a triangulated category with Serre functor, the first condition is equivalent to $\tau\I=\I$ (see Lemma~\ref{lem:tau inv ideal}). Moreover, if $\I$ is an object ideal, the second condition is true provided the first holds(see Proposition~\ref{prop: D-mutation pair}), so we reprove the main theorem in  \cite[Theorem 3.3]{jorgensen2010quotients} (see also \cite[Corollary 4.4]{zhou2018triangulated}). Consider the natural additive quotient category $\HT/\I$, when $\I$ is an object ideal, it is a triangulated category with the given structure (\cite[Theorem 3.3]{jorgensen2010quotients}or \cite[Theorem 3.13]{zhou2018triangulated}), but the Theorem~\ref{thm:gene quot cat} in the present paper proves that this is no longer true for any non-object ideals.

There is an easy observation that a triangulated category $\HT$ having Auslander-Reiten triangles can be understood as the Jacobson radical $\J$ of $\HT$ being an Auslander-Reiten ideal. So the above theorem provides us a way to detect whether $\HT$ has Auslander-Reiten triangles by checking the necessary and sufficient conditions on its radical ideal $\J$: (1) $\J$ is functorially finite, (2) $\Gh_\J=\CoGh_\J$ and (3) $\Gh_\J$-source maps coincide with  $\Gh_\J$-sink maps.  Our method works well when we know the morphisms better than the triangulated structure.

Moreover, the observation also tells us that the classical Auslander-Reiten theory of $\HT$ is a theory on the radical ideal $\J$. It is natural to ask can this theory be generalized to a general Auslander-Reiten ideal? Our second main result is to answer this question. Indeed,  in our generalized Auslander-Reiten theory, the central concept is the ideal mutation triangle which plays a role of the Auslander-Reiten triangle in Auslander-Reiten theory.  Also, other basic concepts, such as source (resp., sink) maps, irreducible morphisms and valued Auslander-Reiten quivers are replaced by $\I$-source (resp., $\I$-sink) maps, left (right) $\I$-irreducible morphisms and valued $\I$-mutation quivers, respectively. This generalized Auslander-Reiten theory shares much in common with the classical one. A triangle is an {\em $\I$-mutation triangle} if the first morphism is an $\I$-source and the second is $\I$-sink. For convenience,
we give a representative property.

\begin{thm}[Theorem~\ref{thm:I-mutation triangle}]
Let $X\stackrel{u}\to Y\stackrel{v}\to Z\stackrel{w}\to X[1]$
be an $\I$-mutation triangle and $W$ an indecomposable object. Write  $Y=\coprod_{i=1}^{n}Y_i^{m_i}$  with $Y_i$ indecomposable and $Y_i\ncong Y_j$ for $i\neq j$. Then
  \begin{enumerate}
    \item $\Irr^-_\I(X, W)\neq 0$ if and only if $W\cong Y_i$ for some $i$;
    \item $\Irr^+_\I(W, Z)\neq 0$ if and only if $W\cong Y_i$ for some $i$;
    \item $m_i=\dim_K\Irr^-_\I(X, W)=\dim_K\Irr^+_\I(W, Z)$.
  \end{enumerate}
\end{thm}

The notations $\Irr^-_\I(X, W)$ and $\Irr^+_\I(W, Z)$ denote the spaces $\I/(\J\I)(X, W)$ and $\I/(\I\J)(W, Z)$, respectively. These spaces have strong relations with $\I$-irreducible morphisms. Indeed, the set $\I\setminus (\J\I)(X, W)$ consists of left irreducible morphisms in $\Hom(X, W)$ and the set $\I\setminus  (\I\J)(W, Z)$ consists of right irreducible morphisms in $\Hom(W, Z)$ (see Lemma~\ref{lem:irred}). From the above theorem, the left and right $\I$-irreducible morphisms appeared in the $\I$-mutation triangle form $K$-bases of $\Irr^-_\I(X, W)$ and $\Irr^+_\I(W, Z)$, respectively. Characterizations of left and right $\I$-irreducible imply that left $\J$-irreducible coincides with right $\J$-irreducible (see Corollary~\ref{cor:J-irreducible}), that is the reason why the classical Auslander-Reiten theory doesn't distinguish between left and right irreducible.

The contents of this paper are organized as follows.
In Section~\ref{sec:preliminaries}, we recall and introduce some basic notation, such as ideal approximations, ghost ideals, ideal torsion pairs and mutation pairs in triangulated categories. Also, some technical results are provided for the preparation of later sections.
In Section~\ref{sec:ideal mutation}, we introduce the notion of ideal mutation pairs in triangulated categories and study their basic properties including a method of detecting Auslander-Reiten ideals. In particular, a method of judging a triangulated category having Auslander-Reiten triangles is given. Furthermore, a discussion on additive quotient categories is given in this section.
In Section~\ref{sec:geme AR theory}, we use the ideal mutation theory to generalize the classical Auslander-Reiten theory. Here, the notion of Auslander-Reiten triangles, a central concept in the classical Auslander-Reiten theory, is replaced by ideal mutation triangles. In Section~\ref{sec:examples}, we give various examples to explain our results, including examples of non-object Auslander-Reiten ideals in triangulated categories, detecting whether a triangulated category has Auslander-Reiten triangles or not and ideal mutation quivers of triangulated categories.
\bigskip

{\bf Conventions.} In this paper, the notion $K$ is assumed to be an algebraic closed field and $\HT$ is assumed to be a $\Hom$-finite Krull-Schmidt triangulated $K$-category. Denote the Jacobson radical of $\HT$ by $\J$ . We assume that a subcategory is closed under isomorphisms, direct sums and direct summands. The composition of $f\in\Hom(X, Y)$ and $g\in\Hom(Y, Z)$ is defined by $gf\in\Hom(X, Z)$. Denote the class morphisms of $\HT$ by $\Mor(\HT)$. An {\em ideal}
$\I$ of $\HT$ is an additive subgroup $\I(X, Y)$ of $\Hom(X, Y)$ such that $hgf\in\I(W, Z)$ for any $f\in\Hom(W, X), g\in\I(X, Y)$ and $h\in\Hom(Y, Z)$ { (or equivalently, $\I(-, -)$ is a sub-bifunctor of $\Hom(-,-)$)}.  Let $\R$ be an another
 ideal of $\HT$, the product $\I\R$ of $\I$ and $\R$ denotes the ideal $\{\Sigma_{k=1}^{n}r_ki_k~|~r_k\in\R, i_k\in\I, n\in\mathbb{N}\}$.  Throughout this paper, $\I$ is assumed to be an ideal of $\HT$.

\section{Preliminaries}\label{sec:preliminaries}
\subsection{Ideal approximations}
Let $T\in\HT$. A morphism $i: X\to T$ in $\I$ is called a {\em right $\I$-approximation} { of $T$} if any other morphism $j: X'\to T$ in $\I$ factors through $i$. Dually, we define {\em left $\I$-approximations}. An ideal $\I$ is called {\em contravariantly(resp., covariantly) finite} if for any object $X\in\HT$, there is a right (resp., left) $\I$-approximation for $X$. We call $\I$ {\em functorially finite} if $\I$ is not only contravariantly finite but also covariantly finite. A map $f: X\to T$ is called {\em $\I$-epic} (resp., $\I$-monic) if each $i: I\to T$ in $\I$ factors through $f$ (resp., each $j: X\to I$ in $\I$ factors through $f$). Obviously, an $\I$-epic (resp., $\I$-monic)  map $f$ is an right (resp., left) $\I$-approximation if and only if $f\in\I$. By $\Ob(\I)$  we denote the full subcategory of $\HT$ with objects $X$ satisfying $\id_X\in\I$. { All definitions in $\HT$ above have equivalent characterizations in the functor category $(\HT^{\mathsf{op}}, \mathsf{Ab})$, the category of contravariant functors from $\HT$ to the category of abelian groups.  A functor $F$ is
said to be {\em finitely generated} if there exists an epimorphism $\Hom(–,T) \twoheadrightarrow F$ for some object $T\in\HT$.

\begin{prop}
 An ideal $\I$ in $\HT$ is contravariantly finite if and only if for each $T\in\HT$ the functor $\I(-, T)$ is a finitely generated functor in $(\HT^{\mathsf{op}}, \mathsf{Ab})$.
\end{prop}
\begin{proof}
$(\Leftarrow)$ Let $T \in \mathcal{T}$. Since $\I(-, T)$ is a finitely generated, there exists an epimorphism
$$
\eta:\Hom(-, X) \twoheadrightarrow \mathcal{I}(-, T) .
$$
This natural transformation may be regarded as a map $\eta:\Hom(-, X) \rightarrow \mathcal{I}(-, T) \subseteq\Hom(-, T)$ so that Yoneda's Lemma implies the existance of $i: X \rightarrow T$ such that $\eta=\Hom(-, i)$. Looking at the $X$ component of this transofrmation shows that $\eta_X\left(\id_X\right)=i \in \mathcal{I}(X, T)$. Because representable objects in $\left(\mathcal{T}^{\mathrm{op}}, \mathrm{Ab}\right)$ are projective, every morphism $i^{\prime}: X^{\prime} \rightarrow T$ in $\mathcal{I}$ yields a morphism in the functor category
$$\xymatrix{
&\Hom(-, X')\ar@{-->}[dl]\ar[d]^{\Hom(-, i')}
\\
\Hom(-, X)\ar@{->>}[r]^{\Hom(-, i)}&\I(-, T)
}$$
that factors as indicated. Thus, $i:X\to T$ is a left $\I$-approximation of $T$. Therefore $\I$ is contravariantly finite.

$(\Rightarrow)$ Let $T\in\HT$ and $i:X\to T$ a left $\I$-approximation of $T$. Then the
$$\Hom(-, X)\stackrel{\Hom(-, i)}{\longrightarrow}\I(-, T)$$
is an epimorphic morphism. This implies that $\I(-, T)$ is finitely generated.
\end{proof}
Similarly, a morphism $i: X \rightarrow T$ is $\mathcal{I}$-epic if the image of the natural transformation $\Hom(-, i):\Hom(-, X) \rightarrow\Hom(-, T)$ in the functor category contains the subfunctor $\mathcal{I}(-, T)$.
}

 Let $\D$ be a subcategory of $\HT$, define
 $$[\D]:=\{f\in \Mor(\HT)~|~f ~\text{factors through some object in}~\D\},$$
  it is obviously an ideal. We call this kind of ideal an {\em object ideal}.
Obviously, an ideal $\I$ is object if and only if $\I=[\Ob\I]$.

A morphism $f:X\to M$ is called {\em left minimal} if it doesn't have a direct summand of the form $0\to M_1(M_1\neq 0)$ (up to an isomorphism). A {\em right minimal morphism} is defined dually. We call a morphism a {\em minimal left (resp., right) $\I$-approximation} if it is both left (resp., right) minimal a and left (resp., right) $\I$-approximation. A minimal left (resp., right) $\I$-approximation is also called an {\em $\I$-source (resp., $\I$-sink) map}.
Let $\A$ be a $\Hom$-finite Krull-Schmidt $K$-category. It is well-known that $f$ is
left (resp., right) minimal if and only if $g$ is isomorphic provided $f=gf$ (resp., $f=fg$) (or see \cite[Lemma~3.4]{plamondon2011cluster}). The proof of the following lemma is direct, we leave it to the reader.

\begin{lem}\label{lem:minimal object}
  Let $\D$ be a subcategory of $\HT$, then
  \begin{enumerate}
    \item if $f:X\to M$ is a minimal left $[\D]$-approximation of $X$, then $M\in\D$;
    \item  if $g:N\to Y$ is a minimal right $[\D]$-approximation of $Y$, then $N\in\D$.
  \end{enumerate}
  \end{lem}

The following lemma seems well-known, but we don't find it in literature. We write a proof here for the convenience of the reader.

\begin{lem}\label{lem:r min VS l min}
  Let
$ X\stackrel{f}\lra M\stackrel{g}\lra Y\stackrel{h}\lra X[1]$
  be a triangle in $\HT$. Then the following are equivalent.
  \begin{enumerate}
    \item $f$ is right minimal;
    \item $g$ is left minimal;
    \item $h$ is a radical morphism.
  \end{enumerate}
\end{lem}
\begin{proof}
(1) $\Rightarrow$ (2)
 Assume $f$ is right minimal. If $g$ isn't left minimal, then $g$ can be written as $\binom{g'}{0}:  M\to Y_1\oplus Y_2$ with $Y_2$ nonzero. Hence $f$ can be written as $(f_1, 0): X_1\oplus Y[-1]\to M$, this contracts to the assumption. Thus $g$ is left minimal. Similarly, we can prove (2) $\Rightarrow$ (1).

 (1) $\Rightarrow$ (3) Let $l: X[1]\to Y$, then we have commutative diagram
 $$\xymatrix{
 X\ar[r]^f\ar@{-->}[d]^{k}
 &M\ar[r]^g\ar@{=}[d]
 &Y\ar[r]^h\ar[d]^{{\id_Y}-lh}
 &X[1]\ar@{-->}[d]^{k[1]}
 \\
 X\ar[r]^f
 &M\ar[r]^g
 &Y\ar[r]^h
 &X[1],
 }$$
since $f$ is right minimal, then { $k$} is isomorphic, and so is $\id_Y-lh$. Thus we know $h$ is a radical morphism.

(3) $\Rightarrow$ (1) If $f$ isn't right minimal, then we can write $f: X=X_1\oplus X_2\stackrel{(f_1, 0)}\to M$ with $X_2\neq 0$. Hence the triangle can be written as
$$\xymatrix{
X_1\oplus X_2\ar[r]^{\quad (f_1, 0)}
& M\ar[r]^{\binom{g_1}{g_2}\quad\quad}
&Y_1\oplus X_2[1]\ar[r]^{\binom{h_1~0~~\quad~}{0~\id_{X_2}}\quad}
&X_1[1]\oplus X_2[1]
}$$
and  $h$ equals $\binom{h_1~0~~\quad~}{0~\id_{X_2}}$ up to an isomorphism. Hence $h$ isn't a radical map, this contracts our assumption. Therefore, $f$ is right minimal.
\end{proof}

\subsection{Ghost ideals}
We define the following two sets of morphisms for $X, Y\in\HT$:
\begin{align*}
   & \Gh_{\I}(X, Y):=\{f\in \Hom(X, Y)~|~fi=0, \forall i\in\I \}, \\
   & \CoGh_{\I}(X, Y):=\{f\in \Hom(X, Y)~|~if=0, \forall i\in\I \}.
\end{align*}

Set $\Gh_{\I}:=\bigcup_{X, Y\in\HT}\Gh_{\I}(X, Y)$ and $\CoGh_{\I}:=\bigcup_{X, Y\in\HT}\CoGh_{\I}(X, Y)$. Obviously, $\Gh_{\I}$ and $\CoGh_{\I}$ are both ideals of $\HT$. Here, we note that $fi=0, \forall i\in\I$ is equivalent to $\Hom_\HT(i, f):=\Hom_\HT(S, f)\circ\Hom_\HT(i, X)=0, \forall i:S\to T\in\I$.
Let $\D$ be a subcategory. Recall that a morphism $f: X\to Y$ is called {\em $\D$-ghost} if $\Hom_\HT(D, f)=0, \forall D\in\D$, dually $f$ is called {\em $\D$-coghost} if $\Hom_\HT(f, D)=0, \forall D\in\D$. Denote the sets of $\D$-ghost morphisms and $\D$-coghost morphisms by $\Gh_{\D}$ and $\CoGh_{\D}$, respectively, one can check easily that $\Gh_{\D}=\Gh_{[\D]}$, $\CoGh_{\D}=\CoGh_{[\D]}$. So our definitions generalize those in  \cite[Section 3.1]{beligiannis2013rigid}.

\begin{lem}\label{lem:ghost-appro}
  Let $C\stackrel{f}\lra A\stackrel{g}\lra B\stackrel{h}\lra C[1]$ be a triangle in $\HT$.  Then
  \begin{enumerate}
    \item $g$ is $\I$-epic if and only if $h\in\Gh_{\I}$. If moreover $g\in\I$, then $g$ is a right $\I$-approximation if and only if $h$ is a left $\Gh_{\I}$-approximation.
    \item $f$ is $\I$-monic if and only if $h\in\CoGh_{\I[1]}$.  If moreover $f\in\I$, then $f$ is a left $\I$-approximation if and only if $h$ is a  right $\CoGh_{\I[1]}$-approximation.
  \end{enumerate}
\end{lem}
\begin{proof}
(1) For any $i\in\I(X, B)$, $i$ factors through $g$  if and only if $hi=0$, this is equivalent to $h\in\Gh_\I$.
Moreover, if $g\in\I$, then for any $j\in\Gh_\I(B, Y)$, we have $jg=0$, hence $j$ factors through $h$, that implies $h$ is a  left $\Gh_{\I}$-approximation.

(2) It is proved { dually}.
\end{proof}

\subsection{Ideal torsion pairs}
A pair $(\I, \R)$ of ideals in $\HT$ is an {\em ideal torsion pair} if
\begin{enumerate}
  \item $\forall i\in\I, r\in\R, \Hom_\HT(i, r)=0$;
  \item For each $T\in\HT$, there is a triangle
  $$\xymatrix{X\ar[r]^f&T\ar[r]^g&Y\ar[r]^h&X[1]},$$
  such that $f\in\I$, and $g\in\R$. We call this triangle a canonical decomposition of $T$ with respect to the ideal torsion pair $(\I, \R)$.
\end{enumerate}

\begin{prop}\label{prop: ideal torsion pair}
  If $(\I, \R)$ is an ideal torsion pair, then
  \begin{enumerate}
    \item $\I=\CoGh_\R$ and $\R=\Gh_\I$.
    \item $\I$ is contravariantly finite and $\R$ is covariantly finite.
  \end{enumerate}
\end{prop}
\begin{proof}
(1)
  Let $a\in\CoGh_\R(S, T)$, and  $X\stackrel{f}\to T\stackrel{g}\to Y\stackrel{h}\to X[1]$ a canonical decomposition of $T$ with respect to $(\I, \R)$. Since $ga=0$, so $a$ factors through $f$, this implies that $a\in\I$. Hence $\CoGh_\R\subset \I$. By the definition of ideal torsion pair $\I\subset \CoGh_\R$,  then we have $\CoGh_\R= \I$.

  (2)  For $T\in\HT$, then $T$ has the canonical decomposition as in (1), in this triangle, the morphisms $f$ is a right $\I$-approximation of $T$ and $g$ is a left $\R$-approximation of $T$.
\end{proof}

 Indeed, by checking easily, a pair $(\X, \Y)$ of subcategories of $\HT$ is a torsion pair in the sense of \cite[Definition 2.2]{iyama2008mutation} if and only if $([\X], [\Y])$ is an ideal torsion pair in $\HT$, a pair $(\P, \I)$ of a subcategory $\P$ and an ideal $\I$ is a projective class in the sense of \cite[Section 2]{christensen1998ideals} if and only if $([\P], \I)$ is an ideal torsion pair in $\HT$, a pair $(\I, \R)$ of ideals is an ideal cotorsion pair in the sense of \cite[Example 3.4.2]{breaz2021ideal} if and only if $(\I, \R[1])$ is an ideal torsion pair in $\HT$.

\begin{lem}\label{lem: determine torsion pair}
The following hold.
\begin{enumerate}
  \item If $\I$ is contravariantly finite, then $(\I, \Gh_\I)$ is an ideal torsion pair.
  \item If $\I$ is covariantly finite, then $(\CoGh_\I, \I)$ is an ideal torsion pair.
\end{enumerate}
\end{lem}
\begin{proof}
We only prove (1), (2) can be proved dually.  Let $T\in\HT$, and $i: I\to T$ a right $\I$-approximation. Extend $i$ to a triangle
$$\xymatrix{
I\ar[r]^{i}
&T\ar[r]^{j}
&J\ar[r]
&I[1].
}$$
Then $j\in\Gh_\I$. Hence $(\I, \Gh_\I)$ is an ideal torsion pair.
\end{proof}

\subsection{Mutations in triangulated categories}
The notation of mutations in triangulated categories was first defined by Iyama and Yoshino \cite{iyama2008mutation}, then a general definition of mutation in extriangulated categories appeared in \cite{liu2013triangulated} (or see \cite{zhou2018triangulated}) which covers the former one. To our interest, we will focus on mutations in triangulated categories.

\begin{defn}(\cite[Definition 2.5]{iyama2008mutation},  \cite[Section 3.2]{liu2013triangulated}, \cite[Definition 3.2]{zhou2018triangulated})\label{defn:object mutation pair}
Let $\D$, $\X$ and $\Y$ be subcategories of $\HT$. The pair $(\X, \Y)$ is called a {\em $\D$-mutation pair} if it satisfies:
\begin{enumerate}
  \item $\D\subset \X\cap \Y$;
  \item For each $X\in\X$, there is a triangle
  $$ X\stackrel{f}\lra M\stackrel{g}\lra Y\stackrel{h}\lra X[1],$$
  where $Y\in\Y$, and $f$ is a left $\D$-approximation, $g$ is a right $\D$-approximation;
  \item For each $Y\in\Y$, there is a triangle
$$ X\stackrel{u}\lra N\stackrel{v}\lra Y\stackrel{w}\lra X[1],$$
  where $X\in\X$, and $v$ is a right $\D$-approximation, $u$ is a left $\D$-approximation.
\end{enumerate}
\end{defn}

\section{Ideal mutations in triangulated categories}\label{sec:ideal mutation}
\subsection{Ideal mutation pairs}\label{subsec: ideal torsion pair}
\begin{defn}\label{defn: I mutation pair}
A pair $(\X, \Y)$ of subcategories of $\HT$ is called an {\em $\I$-mutation pair} if:
\begin{enumerate}
  \item $\Ob{\I}\subset \X\cap \Y$;
  \item For each $X\in\X$, there is a triangle
  $$X\stackrel{f}\lra M\stackrel{g}\lra Y\stackrel{h}\lra X[1],$$
  where $Y\in\Y$, and $f$ is a left $\I$-approximation, $g$ is a right $\I$-approximation;
  \item For each $Y\in\Y$, there is a triangle
$$ X\stackrel{u}\lra N\stackrel{v}\lra Y\stackrel{w}\lra X[1],$$
  where $X\in\X$, and $v$ is a right $\I$-approximation, $u$ is a left $\I$-approximation.
\end{enumerate}
\end{defn}

The following proposition tells us that our definition is a generalization of Definition~\ref{defn:object mutation pair}.

\begin{prop}
  Let $\D$ be a subcategory. Then $(\X, \Y)$ is $\D$-mutation pair if and only if it is a $[\D]$-mutation pair.
\end{prop}
\begin{proof}
  $(\Rightarrow)$ By the definitions of ideal mutation pairs (Definition~\ref{defn: I mutation pair}) and mutation pairs (Definition~\ref{defn:object mutation pair}).

  $(\Leftarrow)$ Let $ X\stackrel{f}\ra M\stackrel{g}\ra Y\stackrel{h}\ra X[1]$ be a triangle with $X\in\X, Y\in\Y$, $f$ a left $[\D]$-approximation and $g$ a right $[\D]$-approximation. Write $f$ as $\binom{f'}{0}:X\to M_0\oplus M_1$ with $f'$ left minimal, then the above triangle is isomorphic to
$$\xymatrix{
X\ar[r]^{\binom{f'}{0}\quad\quad}
&M_0\oplus M_1\ar[r]^{\binom{g_0~0}{0~\id_{M_1}}}
&Y_0\oplus M_1\ar[r]^{(h, 0)}
&X[1].
}$$
Since $g\in[\D]$, $\id_{M_1}\in[\D]$, that is $M_1\in\D$. Because $f'$ is a left minimal $[\D]$-approximation, then $M_0\in\D$ by Lemma~\ref{lem:minimal object}.
  Therefore $M\in \D$. We finish the proof.
\end{proof}

Beyond object ideals, there are also many examples of non-object ideal mutation pairs, we refer the reader to Section~\ref{sec:examples} for examples.

For an ideal $\I$ of $\HT$, put
\begin{align*}
   \X_\I:=&\{X\in\HT~|~\text{there is a triangle}~X\stackrel{f}\lra I^{X}\stackrel{g}\lra M\stackrel{h}\lra X[1]~\text{such that} \\
   &~f~\text{is a left}~\I\text{-approximation~and}~g~\text{is a right}~\I\text{-approximation}\}, \\
  \Y_\I:=&\{Y\in\HT~|~\text{there is a triangle}~N\stackrel{u}\lra I_{Y}\stackrel{v}\lra Y\stackrel{w}\lra N[1]~\text{such that} \\
   &~v~\text{is a right}~\I\text{-approximation~and}~u~\text{is a left}~\I\text{-approximation}\}.
\end{align*}
\begin{rem}\label{rem:ideal mutation pair}
  By Lemma~\ref{lem:I-appro VS GhI-appro}, then $\X_\I$ and $\Y_\I$ can be written equivalently as
  \begin{align*}
   \X_\I=&\{X\in\HT~|~\text{there is a triangle}~X\stackrel{f}\lra I^{X}\stackrel{g}\lra M\stackrel{h}\lra X[1]\\
   &~\text{such that}~f, g\in\I, h\in\Gh_\I\cap \CoGh_{\I[1]}\}, \\
  \Y_\I=&\{Y\in\HT~|~\text{there is a triangle}~N\stackrel{u}\lra I_{Y}\stackrel{v}\lra Y\stackrel{w}\lra N[1] \\
   &~\text{such that}~u, v\in\I,w\in\Gh_\I\cap\CoGh_{\I[1]}\}.
\end{align*}
\end{rem}

The following proposition shows that when we restrict on object ideal, then the above notations coincide with those in \cite[Section 3.1]{beligiannis2013rigid}.

\begin{prop}
  Let $\D$ be a functorially finite subcategory, then
   \begin{enumerate}
     \item $\X_{[\D]}=\{ X\in\HT~|~\CoGh_{\D}(\D[-1], X)=0\}$;
     \item $\Y_{[\D]}=\{ Y\in\HT~|~\Gh_{\D}(Y, \D[1])=0\}$.
   \end{enumerate}
\end{prop}
\begin{proof}
  We only prove (1), (2) can be proved dually.
Let $X\in\HT$ and $X\stackrel{f}\to D^X$  be a minimal left $\D$-approximation, and extend it to a triangle
 $$\xymatrix{
 X\ar[r]^f&D^X\ar[r]^g&M\ar[r]^h&X[1].
 }$$
% By Lemma~\ref{lem:ghost-appro}, we know $h$ is a right $\CoGh_{\D[-1]}$-approximation.

Assume that $X\in \X_{[\D]}$. Let $\al\in \CoGh_{\D}(D[-1], X)$ with $D\in\D$, then $\al$ factors through $h[-1]$ by $\beta$ as in the below diagram
$$\xymatrix{
  &D[-1]\ar@{-->}[ld]_{\beta}\ar^{\al}[d]
  &
  &
  \\
 M[-1]\ar^{\quad h[-1]}[r]
  &X\ar^{f}[r]
  &D^X\ar^{g}[r]
  &M.
  }$$
Since $h\in \Gh_{\D}$ (see Remark~\ref{rem:ideal mutation pair}), then $\al=h[-1] \beta=0$. So $\CoGh_{\D}(D[-1], X)=0$. This means $X$ is also in the right set.

If $X$ satisfies $\CoGh_{\D}(\D[-1], X)=0$. For any $\ga\in\Hom(D, M)$ with $D\in\D$, since the composition $h[-1]\ga[-1]$ is in $\CoGh_{\D}(D[-1], X)$, then $h[-1]\ga[-1]=0$. This implies $h\ga=0$. Hence we have $h\in\Gh_\D$, thus $X\in\X_{[\D]}$. We finish the proof.
\end{proof}

Obviously, $(\X_{\I}, \Y_{\I})$ is the maximal $\I$-mutation pair, that is, if $(\X, \Y)$ is an $\I$-mutation pair, then $\X\subseteq\X_\I$  and  $\Y\subseteq\Y_\I$. And so
  $\HT$ is an $\I$-mutation pair if and only if $\X_\I=\HT=\Y_\I$.

Let $(\X, \Y)$ be an $\I$-mutation pair. We define a map $\Sigma : \X\to \Y$  by fixing a triangle of form
$$X\stackrel{f}\lra I^{X}\stackrel{g}\lra \Sigma X\stackrel{h}\lra X[1] ~ \text{for}~ X\in\X ,$$
where $f$ is a left $\I$-approximation and $g$ is a right $\I$-approximation. In a similar way, we can define a map $\Omega: \Y\to \X$ .
For a morphism $\al$: $X\to X'$  in $\X$ , there exists a commutative diagram
$$\xymatrix{
X\ar^{\al}[d]\ar^{f}[r]
&I^X\ar^{g}[r]\ar@{-->}[d]
&\Sigma X\ar^{h}[r]\ar@{-->}^{\Sigma \al}[d]
& X[1]\ar^{\al[1]}[d]
\\
X'\ar^{f'}[r]
&I^{X'}\ar^{g'}[r]
&\Sigma X'\ar^{h'}[r]
& X'[1].
}$$
Similarly, the action of $\Omega$ on morphisms is defined. { Note that the maps $\Sigma$ and $\Omega$ are not functors in general, but they always induce functors between quotient categories.}

\begin{prop}\textnormal{(Generalization of \cite[Proposition 3.1]{beligiannis2013rigid})}\label{prop:mutaion equivalence}
  Let $(\X, \Y)$ be an $\I$-mutation pair, then
   \begin{enumerate}
     \item The pair of functors { $\overline{\Sigma}:\X/\I\to\Y /\I$ and $\overline{\Omega}:\Y /\I\to \X/\I$ induced from $\Sigma$ and $\Omega$} is an adjoint pair;
     \item { $\overline{\Sigma}$} is an equivalence and { $\overline{\Omega}$} is a quasi-inverse of { $\overline{\Sigma}$}.
   \end{enumerate}
\end{prop}
\begin{proof}
(1) Indeed, { $\overline{\Sigma}$} is independent on the choosing of $f$. If $f$ isn't left minimal, then we can write $f=\binom{f_1}{0}: X\to I_1\oplus I_2$ with $f_1$ left minimal and $I_2\neq 0$, and the triangle is isomorphic to
$$\xymatrix{
X\ar[r]^{\binom{f_1}{0}\quad}&I_1\oplus I_2\ar[r]^{\binom{g_1~0}{0~\id_{I_2}}}&M\oplus I_2\ar[r]&X[1]
},$$
since $g\in\I$, then $\id_{I_2}\in\I$ and so $I_2\in\Ob\I$, hence $\Sigma X\simeq M$ in $\X$.
 Let $\al\in\I$, then it factors through $f$, this implies that $\Sigma \al$ factors through $g'$, thus we have $\Sigma \al\in\I$. Therefore { $\overline{\Sigma}$} is well-defined. The proof for  { $\overline{\Omega}$} is similar.

Let $M\in\Y$, $N\in\X$. We define
\begin{align*}
  \Phi_{M, N}:~\X/\I(\Sigma M, N)&\longrightarrow  \Y/\I(M, \Omega N) \\
  \bar{\al} &\longmapsto \bar{\beta}
\end{align*}
where $\al$ and $\beta$ are from
  $$\xymatrix{
  M\ar@{-->}[d]^\beta\ar^{f}[r]
  &I^M\ar@{-->}[d]\ar^{g}[r]
  &\Sigma M\ar[d]^{\al}
  \\
  \Omega N\ar^{f'}[r]
  &I_N\ar[r]^{g'}
  &N.
  }$$
$\Phi_{M, N}$ is well-defined since
  $\al$ factors through $g'$ if and only if $\beta$ factors through $f$. $\Phi_{M, N}$ is an isomorphism since it has an obvious inverse
  \begin{align*}
  \Psi_{M, N}:~\Y/\I(M, \Omega N)&\longrightarrow  \X/\I(\Sigma M, N)\\
  \bar{\beta} &\longmapsto \bar{\al}.
\end{align*}

  (2)  Denote the unit and counit of the adjoint pair ${ (\overline{\Sigma}, \overline{\Omega})}: \X/\I\to \Y/\I$ by $\eta$ and $\varepsilon$, respectively.
   Let $X\in\X$ and $\sigma_{\Sigma X}: I_{\Sigma X}\to \Sigma X$ be a minimal right $\I$-approximation, then there is the commutative diagram
$$\xymatrix{
X\ar^{f}[rr]\ar_{\binom{g_1}{g_2}}[d]
&
&I^X\ar^{g}[rr]\ar^{\simeq}[d]
&
&\Sigma X\ar@{=}[d]
\\
Z\oplus\Omega\Sigma X\ar^{\binom{\id_Z~0}{\quad~~~0~\delta_{\Sigma X}}}[rr]
&
&Z\oplus I_{\Sigma X}\ar^{(0, \sigma_{\Sigma X})}[rr]
&
&\Sigma X.
}$$
where $Z\in \Ob\I$.  Hence $\eta_X=\overline{g_2}=\overline{\binom{g_1}{g_2}}$ is invertible in $\HT/\I$. Similarly, we can prove that $\varepsilon$ is an isomorphism. Hence { $\overline{\Sigma}$} is an equivalent functor from $\X/\I$ to $\Y/\I$, and { $\overline{\Omega}$} is a quasi-inverse.
\end{proof}

\subsection{Characterizations of Auslander-Reiten ideals}
A functorially finite ideal $\I$ is called an {\em Auslander-Reiten ideal} of $\HT$ if $(\HT, \HT)$ is an $\I$-mutation pair. For a  triangulated category having Auslander-Reiten triangles, the Jacobson radical is an Auslander-Reiten ideal.
In this subsection, we will give a characterization of Auslander-Reiten ideals.

\begin{lem}\label{lem:forms of I-muta tri}
 Let  $X\stackrel{f}\to Y\stackrel{g}\to Z\stackrel{h}\to X[1]$ be a
 triangle with $f$ a left $\I$-approximation and $g$ a right $\I$-approximation. Then it is a sum of triangles of the forms:
  \begin{enumerate}
    \item $X_1\stackrel{u}\to I\stackrel{v}\to Z_1\stackrel{w}\to X[1]$, $X_1, Z_1\notin\Ob(\I)$ are indecomposable, $u$ is an $\I$-source map and $v$ is an $\I$-sink map;
    \item $X_2\stackrel{\id_{X_2}}\to X_2\stackrel{0}\to 0\stackrel{0}\to X_2[1]$, $X_2\in\Ob(\I)$ and is indecomposable;
    \item $0\stackrel{0}\to X_3\stackrel{\id_{X_3}}\to X_3\stackrel{0}\to 0$, $X_3\in\Ob(\I)$ and is indecomposable.
  \end{enumerate}
\end{lem}
\begin{proof}
If $f$ isn't a left minimal morphism, then there are a decomposition $M=M_0\oplus M_1$ with $M_1\neq 0$ and a morphism $\binom{f'}{0}: X\to M_0\oplus M_1$ with $f'$ left minimal. The given triangle is isomorphic to
$$\xymatrix{
X\ar[r]^{\binom{f'}{0}\quad\quad}
&M_0\oplus M_1\ar[r]^{\binom{g_0~0}{0~\id}}
&Y_0\oplus M_1\ar[r]^{(h_0, 0)}
&X[1]
}$$
which is a sum of the triangle $0\stackrel{0}\to M_1\stackrel{\id_{M_1}}\to M_1\stackrel{0}\to 0$ with $M_1\in \Ob(\I)$ and the triangle
$$\xymatrix{
X\ar[r]^{f'}
&M_0\ar[r]^{g_0}
&Y_0\ar[r]^h
&X[1].
}\quad (*)$$
If $g_0$ isn't right minimal, then similarly, there is a direct summand $X_1$ of both $X$ and $M_0$, hence $(*)$ is a sum of the triangle $X_1\stackrel{\id_{X_1}}\to X_1\stackrel{0}\to 0\stackrel{0}\to X_1[1]$ with $X_1\in\Ob(\I)$ and the triangle
$$\xymatrix{
X_2\ar[r]^{f_2}
&M'_2\ar[r]^{g'_2}
&Y_2\ar[r]^h
&X_2[1]
}\quad (**)$$
with $f_2$ left minimal and $g'_0$ right minimal, the triangle $(**)$ is a sum of triangles of the first kind.
\end{proof}

With this lemma, we can replace the assumptions "$X\in\HT$"  and "$Y\in\HT$" in Definition~\ref{defn: I mutation pair} with "$X\in\ind \HT\setminus\Ob\I$" and "$Y\in\ind \HT\setminus\Ob\I$", respectively. So if the Jacobson radical $\J$  of $\HT$ is an Auslander-Reiten ideal, then $\HT$ has Auslander-Reiten triangles.

\begin{lem}\label{lem:I-appro VS GhI-appro}
   Let $X\stackrel{f}\to Y\stackrel{g}\to Z\stackrel{h}\to X[1]$ be a
 triangle and $\I$ a functorially finite ideal. Then
\begin{enumerate}
  \item $g$ is a right $\I$-approximation (resp., $\I$-sink map) if and only if $h$ is a  left $\Gh_\I$-approximation (resp., $\Gh_\I$-source map).
  \item $f$ is a left $\I$-approximation (resp., $\I$-source map) if and only if $h$ is a right $\CoGh_{\I[1]}$-approximation (resp., $\CoGh_{\I[1]}$-sink map).
\end{enumerate}
\end{lem}
\begin{proof}
 (1)
  $(\Rightarrow)$  By Lemma~\ref{lem:ghost-appro}(1) and Lemma~\ref{lem:r min VS l min}.

  $(\Leftarrow)$ Let $j\in\Gh_\I$, since $h$ is a left $\Gh_\I$-approximation, then $j$ factors through $h$. So $jg=0$, this implies $g\in\I$ by Lemma~\ref{lem: determine torsion pair}(1). Moreover, $g$ is a right $\I$-approximation by Lemma~\ref{lem:ghost-appro}(1) and is right minimal provided $h$ is left minimal by Lemma~\ref{lem:r min VS l min}.

  (2) It is proved dually.
\end{proof}

\begin{thm}\label{thm:char of AR ideal}
  Let $\I$ be a functorially finite ideal in $\HT$. Then $\I$ is an Auslander-Reiten ideal if and only if
  \begin{enumerate}
    \item $\Gh_\I=\CoGh_{\I[1]}$, and moreover
    \item for $X\in\ind\HT\setminus\Ob\I$, each  $\Gh_\I$-source map of $X$ is a  $\Gh_\I$-sink map and each $\Gh_\I$-sink map of $X[1]$ is a $\Gh_\I$-source map.
  \end{enumerate}
\end{thm}
\begin{proof}
$(\Rightarrow)$ Since $\Gh_\I$ is covariantly finite, then to prove $\Gh_\I\subseteq \CoGh_{\I[1]}$, it is sufficient to prove each $\Gh_\I$-source map is in $\CoGh_{\I[1]}$. Let $h: Z\to X[1]$ be a $\Gh_\I$-source map with $Z$ being indecomposable and not in $\Ob\I$. Extend $h$ to a triangle
$$\xymatrix{
X\ar[r]^f &Y\ar[r]^g & Z\ar[r]^h& X[1].
}$$
  Then $g$ is an $\I$-sink map by Lemma~\ref{lem:I-appro VS GhI-appro}. Since $\I$ is an Auslander-Reiten ideal, then $f$ is a left $\I$-approximation. This implies that $h$ is a right $\CoGh_{\I[1]}$-approximation by Lemma~\ref{lem:ghost-appro}. Hence we have $\Gh_\I\subseteq \CoGh_{\I[1]}$. Similarly, we can prove $\CoGh_{\I[1]}\subseteq\Gh_\I$. Therefore $\Gh_\I= \CoGh_{\I[1]}$. Moreover, $h$ is right minimal since $Z$ is indecomposable. Thus, $h$ is also a $\Gh_\I$-sink map. For the inverse case, since $\Gh_\I= \CoGh_{\I[1]}$, hence $\Gh_\I$ is also contravariantly finite, by a similar proof, we can prove that the $\Gh_\I$-sink map of $X[1]$ is a $\Gh_\I$-source map.

  $(\Leftarrow)$ It is sufficient to prove $X_\I=\Y_\I=\HT$. Let $X\in\ind \HT\setminus\Ob\I$, and $f: X\to Y$ an $\I$-source map. Extend $f$ to a triangle
$$\xymatrix{
X\ar[r]^f &Y\ar[r]^g & Z\ar[r]^h& X[1].
}$$
  It follows from Lemma~\ref{lem:I-appro VS GhI-appro} that $h$ is a  $\CoGh_{\I[1]}$-sink map of $X[1]$, then by the assumption (1) it is also a $\Gh_\I$-sink map of $X[1]$, and by the assumption (2) it is a $\Gh_\I$-source map of $Z$. This implies that  $g\in\I$ by Lemma~\ref{lem:ghost-appro}. Therefore $X\in\X_\I$, and so $\X_\I=\HT$. Similarly, we can prove $\Y_\I=\HT$. We finish the proof.
  \end{proof}

  The second condition can't be derived from the first, see Example~\ref{exam:exsmple no tri}(ii). The following corollary provides us with necessary and sufficient conditions to investigate whether a triangulated category has Auslander-Reiten triangles.

\begin{cor}\label{cor:detect AR for J}
  $\J$ is an Auslander-Reiten ideal if and only if $\J$ satisfies
  \begin{enumerate}
    \item $\J$ is functorially finite,
    \item $\Gh_\J=\CoGh_\J$, and
    \item $\Gh_\J$-source maps coincide with  $\Gh_\J$-sink maps.
  \end{enumerate}
\end{cor}

A $K$-linear autofunctor $\mathbb{S}: \HT\to\HT$ is called a {\em Serre functor} of $\HT$ if there exists a functorial isomorphism $\Hom_{\HT}(A, B)\simeq D\Hom_{\HT}(B, \mathbb{S}A)$ for any $A, B\in\HT$, where $D=\Hom_K(-, K)$. It follows from \cite[Theorem I.2.4]{reiten2002noether} that $\HT$ has Auslander-Reiten triangles if and only if $\HT$
 has a Serre functor. In this case, the translation $\tau=\mathbb{S}[-1]$. An ideal $\I$ is called {\em $\tau$-stable} if $\tau \I=\I$, where $\tau_\I=\{\tau(f)~|~f\in\I\}$. Obviously, the Jacobson radical $\J$ is $\tau$-stable.

\begin{lem}\label{lem:tau inv ideal}
  Let $\HT$ be a triangulated category with  Serre functor $\mathbb{S}$, and $\I$ a functorially finite ideal of $\HT$. Then $\I$ is $\tau$-stable if and only if  $\Gh_\I=\CoGh_{\I[1]}$.
\end{lem}
\begin{proof}
If $\I$ is $\tau$-stable, then there are the following equivalences
\begin{align*}
  f\in\Gh_\I \Leftrightarrow&\Hom_\HT(i, f)=0, \forall i\in\I ,\\
   \Leftrightarrow& D\Hom_\HT(f, \mathbb{S}i)=0, \forall i\in\I ,\\
  \Leftrightarrow& \Hom_\HT(f, \tau i[1])=0, \forall i\in\I,\\
  \Leftrightarrow& f\in\CoGh_{\I[1]}.
\end{align*}
Thus we have $\Gh_\I=\CoGh_{\I[1]}$.

Conversely, for any $i\in\I$, $f\in \Gh_\I=\CoGh_{\I[1]}$.  Then
$$D\Hom_\HT(f, \mathbb{S}i)=\Hom_\HT(i, f)=0,$$
so $\Hom_\HT(f, \tau i[1])=0$. Since $(\CoGh_{\I[1]}, \I[1])$ is an ideal torsion pair, thus it follows from Proposition~\ref{prop: ideal torsion pair} that $\tau i[1]\in \I[1]$. Therefore $\tau \I=\I$. We finish the proof.
\end{proof}

\begin{cor}\label{cor: char of mut pair}
  Let $\HT$ be a triangulated category with  Serre functor $\mathbb{S}$, and $\I$ a functorially finite ideal of $\HT$. If $\I$ is an Auslander-Reiten ideal, then $\I$ is $\tau$-stable.
\end{cor}

In general, $\tau$-stable ideal isn't an Auslander-Reiten ideal, see a counterexample in Example~\ref{exam:exsmple no tri}(ii). But if the given ideal is object, then they are equivalent.

\begin{prop}\label{prop: D-mutation pair}
  Let $\D$ be a functorially finite subcategory of $\HT$.
  \begin{enumerate}
    \item $\tau \D=\D$ if and only if $[\D]$ is $\tau$-stable;
    \item With the conditions in (1), then for $X\in\ind \D$, each $\Gh_{[\D]}$-source map of $X$ is a $\Gh_{[\D]}$-sink map and each $\Gh_{[\D]}$-sink map of $X[1]$ is a $\Gh_{[\D]}$-source map.
  \end{enumerate}
\end{prop}
\begin{proof}
(1)
  $(\Rightarrow)$ Let $f:X\to Y$ be a morphism in $[\D]$, then $\tau f$ factors through $\tau D$. Since $\tau D\in\D$, so $\tau f$ factors through $\D$, i.e. $\tau f\in[\D]$, thus we have $\tau [\D]\subseteq [\D]$. On the other side, $\tau^{-1}f$ factors through $\tau^{-1} D'\in\D$. Hence $[\D]\subseteq \tau [\D]$. Therefore $[\D]= \tau [\D]$.

  $(\Leftarrow)$ Since $\tau \id_D=\id_{\tau D}\in [\D]$ for $D\in\D$, so $\tau \D\subseteq \D$. On the other side, for any $D\in\D$, $\id_D\in[\D]=[\tau \D]$, then $D\in \tau \D$. Thus we have $\tau \D=\D$.
(2) Assume that $h: Z\to X[1]$ is a $\Gh_\D$-source map with $Z\in\ind\HT\setminus\D$. Extend $h$ to a triangle
$$X\stackrel{f}\to Y\stackrel{g}\to Z\stackrel{h}\to X[1].$$
 Then $g$ is a $[\D]$-sink map by Lemma~\ref{lem:I-appro VS GhI-appro}, and $Y\in\D$ by Lemma~\ref{lem:minimal object}. So $f\in [\D]$.  Since $h\in\Gh_\D=\CoGh_{\D[1]}$ (by Lemma~\ref{lem:tau inv ideal}), then it follows from Lemma~\ref{lem:ghost-appro}(2) that $h$ is a right $\Gh_\D$-approximation and also a $\Gh_\D$-sink map since $Z$ is indecomposable. Similarly, we can prove the dual statement.
\end{proof}

\begin{cor}{$($\cite[Theorem 3.3]{jorgensen2010quotients}, \cite[Theorem 4.3]{zhou2018triangulated}$)$}\label{cor: object mutation}
  Let $\D$ be a functorially finite subcategory of $\HT$. Then $(\HT, \HT)$ is a $\D$-mutation pair if and only if $\tau \D=\D$.
 \end{cor}

\subsection{On additive quotient categories}\label{subsect: subfactor}
Let $(\HT, \HT)$ be an $\I$-mutation pair. Let $f:X\to Y$ be an $\I$-monic morphism, then there is  the following commutative diagram
$$\xymatrix{
X\ar@{=}[d]\ar[r]^f&Y\ar@{-->}[d]\ar[r]^g&Z\ar@{-->}[d]^d\ar[r]^h&X[1]\ar@{=}[d]
\\
X\ar[r]^u&I_X\ar[r]^v&\Sigma X\ar[r]^w&X[1]
}$$
where the first row is the triangle extended from $f$ and the second row is a triangle given by $\I$-mutation. From this diagram, we obtain a sequence in $\HT/\I$
$$\xymatrix{
X\ar[r]^{\bar{f}}&Y\ar[r]^{\bar{g}}&Z\ar[r]^{\bar{d}}&\Sigma X.
}$$
We denote the class of such kind sequences by $\E_\I$.

It is known that $\HT/\I$  is a $\Hom$-finite Krull-Schmidt $K$-category with auto-equivalence $\Sigma$ (See Proposition~\ref{prop:mutaion equivalence}). And in the case of $\I$ being an object ideal, $(\HT, \E_\I, \Sigma)$ is a triangulated category (See \cite[Theorem 3.3]{jorgensen2010quotients} or \cite[Theorem 3.13]{zhou2018triangulated}). But in the case of $\I$ being non-object, things become very different.

\begin{exam}\label{exam:note on subfactor}
  (1) Let $\HT$ be a triangulated category with Auslander-Reiten triangles, that is, $(\HT, \HT)$ is a $\J$-mutation pair. The additive quotient category $\HT/\J$ is a triangulated category with split triangles and the suspension functor $\Sigma$ determined by $\J$. Note that the constructed class $\E_\J$ doesn't equal the class of split triangles. In fact, an Auslander-Reiten triangle in $\HT$
  $$\xymatrix{
  X\ar[r]^f&Y\ar[r]^g&Z\ar[r]^h&X[1]
  }$$
  induces the following sequence
  $$\xymatrix{
  X\ar[r]^0&Y\ar[r]^0&\Sigma X\ar[r]^\id&\Sigma X
  }$$
in $\HT/\J$  which belongs to $\E_\J$, but it is not a split triangle in $\HT/\J$.

  (2) In Example~\ref{exam:exsmple no tri}(i). Note that $\bar{\al}$ in $\HT/\I$ admits a minimal quasi-cokernel $\bar{\beta}$, if $\HT/\I$ has a triangulated structure, then there exists a triangle of the form
$$X\stackrel{\bar{\al}}\to Y\stackrel{\bar{\beta}}\to X\stackrel{\bar{h}}\to Y$$
where $\bar{h}$ equals $\bar{\al}$ up to an isomorphism in $\HT/\I$, but we have the following diagram
$$\xymatrix{
X\ar[r]^{\bar{\al}}\ar@{=}[d]
&Y\ar[r]^{\bar{\beta}}\ar[d]^{\bar{\beta}}
&X\ar[r]^{\bar{h}}\ar@{-->}[d]^{\binom{1}{0}}
&Y\ar@{=}[d]
\\
X\ar[r]^0
&X\ar[r]^{\binom{1}{0}\quad}\ar[d]^{\bar{h}}
&X\oplus Y\ar[r]^{\quad(0, 1)}\ar@{-->}[d]
&Y\ar[d]
\\
&Y\ar@{=}[r]\ar[d]^{-\bar{\beta}}
&Y\ar[r]\ar@{-->}[d]^0
&X
\\
&X\ar[r]^{\bar{\al}}
&Y
&
}$$
where the square in the upper right corner isn't commutative. Hence the octahedral axiom fails. This means that $\HT/\I$ has no triangulated structure associated with the autoequivalence functor $\Sigma$.
\end{exam}

Note that in the object case (i.e. $\I=[\D]$), the fact $\E_\I$ is a triangulated structure on $\HT/\I$ relies on that each morphism in $\HT/\I$ has an $\I$-monic (or $\I$-epic) representation. But there is no such result in the setting of non-object. In the Example~\ref{exam:note on subfactor}(2), $\bar{\al}: X\to Y $ in $\HT/\I$ has no $\I$-monic representation. Moreover, the following theorem tells us that only when $\I$ is an object ideal the constructed class $\E_\I$ is a triangulated structure associated with the shift functor $\Sigma$.

\begin{thm}\label{thm:gene quot cat}
  Let $\I$ be an Auslander-Reiten ideal. Then $(\HT/\I, \E_\I, \Sigma)$ is a triangulated category if and only if $\I$ is an object ideal.
\end{thm}
\begin{proof}
  The sufficiency is due to  \cite[Theorem 3.3]{jorgensen2010quotients} or \cite[Theorem 3.13]{zhou2018triangulated}.

  For necessity, let $X\in\HT$, then by the definition of ideal mutation pairs there is a triangle in $\HT$
    $$\xymatrix{
  X\ar[r]^f&Y\ar[r]^g&\Sigma X\ar[r]^h&X[1]
  }$$
  with $f$ a left $\I$-approximation and $g$ a right $\I$-approximation. Moreover, we have the following commutative diagram
  $$\xymatrix{
  X\ar@{=}[d]\ar[r]^f&Y\ar@{=}[d]\ar[r]^g&\Sigma X\ar@{=}[d]\ar[r]^h&X[1]\ar@{=}[d]
  \\
  X\ar[r]^f&Y\ar[r]^g&\Sigma X\ar[r]^h&X[1]
  }$$
  Since $f$ is $\I$-monic, then the above diagram induces the following sequence
   $$\xymatrix{
  X\ar[r]^0&Y\ar[r]^0&\Sigma X\ar[r]^{\id\quad\quad}&\Sigma X\quad(\#)
  }$$
 in $\HT/\I$ which belongs to $\E_\I$, i.e. $(\#)$ is a triangle in $\HT/\I$. Hence $Y\in\Ob(\I)$, and the left $\I$-approximation $f$ factors through an object in $\Ob(\I)$. Since each morphism in $\I$ factors a left  $\I$-approximation. Thus each morphism in $\I$ factors through an object in $\Ob(\I)$. This means $\I$ is object. We finish the proof.
\end{proof}

\section{Generalised Auslander-Reiten theory}\label{sec:geme AR theory}
In this section, we generalize the classical Auslander-Reiten theory of triangulated categories by using ideal mutations. All proofs in this section are similar to those in the classical Auslander-Reiten theory (for instance, see \cite{assem2020basic, assem2006elements, auslander1995representation}).

\subsection{Ideal mutation triangles}
A triangle is called an {\em $\I$-mutation triangle} in $\HT$ if it is one of the forms in Lemma~\ref{lem:forms of I-muta tri}. In particular, we call the second and third kinds trivial $\I$-mutation triangles.
We say $\HT$ has $\I$-mutation triangles if for each nonzero object $Z\in\ind\HT$, there is an $\I$-mutation triangle in which $Z$ lies in the third item (Indeed, it equals to say there is an $\I$-mutation triangle in which $Z$ lies in the first item). According to Lemma~\ref{lem:forms of I-muta tri},  $\HT$ has $\I$-mutation triangles if and only if $\I$ is an Auslander-Reiten ideal. Note that the usual source (resp., sink) maps coincide with the $\J$-source (resp., $\J$-sink) maps and the Auslander-Reiten triangles coincide with the $\J$-mutation triangles here.

\subsection{\texorpdfstring{$\I$}{I}-irreducible morphisms}
A morphism $h: X\to Y$ is {\em left(resp., right) $\I$-irreducible} if
\begin{enumerate}
  \item $h\in\I$ ;
  \item For any decomposition $h=h_2h_1$, if $h_1\in\I$, then $h_2$ is split epic(resp., if $h_2\in\I$, then $h_1$ is split mono).
\end{enumerate}

A morphism is {\em $\I$-irreducible} if it is both left and right $\I$-irreducible.  Obviously, the concept of irreducible in algebraic representation theory  coincides with $\J$-irreducible here. When $\I=[\D]$, then each $h: X\to Y\in[\D]$ can be decomposed to
$$\xymatrix{
X\ar[rr]^h\ar[dr]_{h_1}&&Y
\\
&D.\ar[ur]_{h_2}&
}$$
If $h$ is left(resp., right) $[\D]$-irreducible, then $h_2$ is split epic, that is $Y\in\add\D$(resp., $X\in\add\D$).
\begin{lem}\label{lem:irrd of obj in I}
Let $f:X\to Y$ be a morphism.
\begin{enumerate}
  \item If $X$ is an indecomposable object in $\Ob(\I)$, then $f$ is an isomorphism if and only if $f$ is left $\I$-irreducible;
  \item If $Y$ is an indecomposable object in $\Ob(\I)$, then $f$ is an isomorphism if and only if $f$ is right $\I$-irreducible;
  \item If $X, Y$  are indecomposable objects in $\Ob(\I)$, then $f$ is an isomorphism if and only if $f$ is $\I$-irreducible.
\end{enumerate}
\end{lem}
\begin{proof}
We only prove (1), (2) can be proved similarly and (3) is directly from (1) and (2).
If $f$ is an isomorphism, then it is obvious an $\I$-irreducible. Assume $f$ is left $\I$-irreducible, since $f=f\id_X$ and $\id_X\in\I$, then $f$ is split epic, because $X$ is indecomposable, so $f$ is an isomorphism. This finishes the proof.
  \end{proof}

\begin{lem}
Each left $\I$-irreducible morphism is left minimal and each right $\I$-irreducible morphism is right minimal.
\end{lem}
\begin{proof}
Let $f: X\to Y$ be a left $\I$-irreducible morphism, assume $f$ isn't left minimal, write $f=\binom{f'}{0}$ and $Y=Y_1\oplus Y_2$ with $Y_2\neq 0$, then there is the following commutative diagram
  $$\xymatrix{
  X\ar[r]^{\binom{f'}{0}\quad}\ar[d]_{f'}
  &Y_1\oplus Y_2
  \\
  Y_1,\ar[ur]_{\binom{\id}{0}}
  &
  }$$
  since ${\binom{\id}{0}}$ is not split epic, then $f'\notin \I$, this conflicts the definition of $f$.  Therefore $f$ must be left minimal. The dual statement is proved similarly.
\end{proof}

\begin{lem}\label{lem: source VS irreducible}
  Each $\I$-source (resp., $\I$-sink) map is left (resp., right) $\I$-irreducible.
\end{lem}
\begin{proof}
     Let $f: X\to Y$ be an $\I$-source map. Assume $f$ has a decomposition $f=gi$ with $i:X\to Z$ in $\I$, then there exists $h:Y\to Z$ such that $i=hf$, thus we have $f=gi=ghf$. Since $f$ is left minimal, then $gh$ is an isomorphism, this implies  $g$ is split epic. Thus we have $f$ is left $\I$-irreducible. Similarly, we can prove each $\I$-sink map is right $\I$-irreducible.
\end{proof}

  \begin{prop}
  Let $X$ be an indecomposable object not in $\Ob(\I)$.
  \begin{enumerate}\label{prop: source VS irreducible}
    \item Let $f:X\to Y$ be an $\I$-source map of $X$. Then, $f': X\to Y'$ is left $\I$-irreducible if and only if $Y'\neq 0$ and there exist a decomposition $Y=Y'\oplus Y''$ (up to an isomorphism) and a morphism $f'': X\to Y''$ such that
        $\binom{f'}{f''}:X\to Y$ is an $\I$-source map of $X$.
    \item Let $g:Z\to X$ be an $\I$-sink map of $X$. Then, $g': Z'\to X$ is left $\I$-irreducible if and only if $Z'\neq 0$ and there exist a decomposition $Z=Z'\oplus Z''$ (up to an isomorphism) and a morphism $g'': Z''\to X$ such that
        $(g', g''):Z'\oplus Z\to X$ is an $\I$-sink map of $X$.
  \end{enumerate}
  \end{prop}
  \begin{proof}
We only prove (1), the proof of (2) is dual.

{\em Necessity.} Since $f'\in\I$, then $f'$ factors through $f$ by a morphism $u:Y\to Y'$. Because $f'$ is left $\I$-irreducible and $f\in\I$, thus $u$ is split epic. This implies the statement.

{\em Sufficiency.} Since $\binom{f'}{f''}\in\I$, then $f'=(\id, 0)\binom{f'}{f''}$ and $f''=(0, \id)\binom{f'}{f''}$ are both in $\I$. Assume $f'$ has a decomposition $f'=ba$ with $a:X\to U\in\I$, then $\binom{f'}{f''}$ has a decomposition as in the following commutative diagram
    $$\xymatrix{
    X\ar[rr]^{\binom{f'}{f''}\quad}\ar[dr]_{\binom{a}{f''}\quad}
    &&Y'\oplus Y''
    \\
    &U\oplus Y''\ar[ur]_{\quad b\oplus\id}
    &
    }$$
    since $\binom{a}{f''}\in\I$ and $\binom{f'}{f''}$ is left $\I$-irreducible by Lemma~\ref{lem: source VS irreducible}, then $b\oplus\id$ is split epic, and hence  $b$ is split epic. Therefore $f'$ is left $\I$-irreducible. We finish the proof.
  \end{proof}

  There is the following consequence of the above proposition.

  \begin{cor}\label{cor:split irre}
  Let $X$ be an indecomposable object not in $\Ob(\I)$.
    \begin{enumerate}
      \item Let $f:X\to Y$ be an $\I$-source map of $X$ and $p: Y\to Y'$ a split epic. Then $pf$ is left $\I$-irreducible.
      \item Let $g:Z\to X$ be an $\I$-sink map of $X$ and $i:Z'\to Z$ a split mono. Then $gi$ is right $\I$-irreducible.
    \end{enumerate}
  \end{cor}

\begin{lem}\label{lem:irred}
Let $f: X\to Y$ be a morphism between indecomposable objects $X$ and $Y$.  Then the following hold.
  \begin{enumerate}
    \item $f$ is left $\I$-irreducible if and only if $f\in\I(X, Y)\setminus \J\I(X, Y)$;
    \item $f$ is right $\I$-irreducible if and only if $f\in\I(X, Y)\setminus \I\J(X, Y)$;
    \item $f$ is $\I$-irreducible if and only if $f\in\I(X, Y)\setminus(\J\I(X, Y)\bigcup\I\J(X, Y))$.
  \end{enumerate}
\end{lem}
\begin{proof}
  We only prove (1), (2) and (3) are left to the reader.
  $(\Rightarrow)$ Since $f$ is a left $\I$-irreducible, then $f\in\I(X, Y)$. If $f\in \J\I(X, Y)$, then $f=\Sigma_{i=1}^{n}h_ig_i$, where $g_i:X\to Z_i$ is in $\I$ and $h_i: Z_i\to Y$ is in $\J$. Indeed, $f$ can be written as
  $$f=(g_1,g_2, \cdots, g_n)\circ(h_1, h_2, \cdots, h_n)^t$$
  where $(g_1,g_2, \cdots, g_n)\in\I$ and $(h_1, h_2, \cdots, h_n)^t\in\J$.
  This conflicts with the definition of left $\I$-irreducible morphisms.

  $(\Leftarrow)$ Since $f\notin \J\I(X, Y)$, then for any decomposition $f=hg$ with $g\in\I$, we have $h\notin \J$, this implies $h$ is split epic. Hence $f$ is left $\I$-irreducible.
\end{proof}

Let $X, Y$ be indecomposable objects in $\HT$, denote
  \begin{align*}
     & {{\Irr}^{-}_\I}(X, Y):=\I/(\J\I)(X, Y), \\
     & {{\Irr}^{+}_\I}(X, Y):=\I/(\I\J)(X, Y),\\
     & {{\Irr}_\I}(X, Y):=\I/(\J\I+\I\J)(X, Y).
  \end{align*}

The following corollary tells us the reason why the classical theory doesn't distinguish between the left and right sides.

\begin{cor}\label{cor:J-irreducible}
  Let $\J$ be the Jacobson radical of $\HT$. Then
  \begin{enumerate}
    \item $f$ is left $\J$-irreducible if and only if $f$ is right $\J$-irreducible;
    \item $\Irr^-_\J(X, Y)=\Irr^+_\J(X, Y)=\Irr_\J(X, Y)$.
  \end{enumerate}
\end{cor}

\subsection{Valued \texorpdfstring{$\I$}{I}-mutation quivers}

\begin{thm}\label{thm:I-mutation triangle}
Let $X\stackrel{u}\to Y\stackrel{v}\to Z\stackrel{w}\to X[1]$
be a non-trivial $\I$-mutation triangle and $W$ an indecomposable object. Let $Y=\coprod_{i=1}^{n}Y_i^{m_i}$  with $Y_i$ indecomposable and $Y_i\ncong Y_j$ for $i\neq j$. Then
  \begin{enumerate}
    \item $\Irr^-_\I(X, W)\neq 0$ if and only if $W\cong Y_i$ for some $i$;
    \item $\Irr^+_\I(W, Z)\neq 0$ if and only if $W\cong Y_i$ for some $i$;
    \item $m_i=\dim_K\Irr^-_\I(X, W)=\dim_K\Irr^+_\I(W, Z)$.
  \end{enumerate}
\end{thm}
\begin{proof}
  (1)
  The sufficiency is direct from Proposition~\ref{prop: source VS irreducible}. For necessity. Let $f\neq 0\in\Irr^-_\I(X, W)$, then there exists $g:Y\to W$ such that
   $f=gu$. Since $u\in\I$, then $g$ is split epic. Therefore $W\cong Y_i$ for some $i$.

   (2) The proof is similar to (1).

   (3) Denote $u=(u_{ij})^t_{1\leq i\leq n, 1\leq j\leq m_i}$, where $(u_{i_1}, \cdots, u_{im_i})^t: X\to Y_i^{m_i}$.
 It is sufficient to show that the set $\{\bar{u}_{ij}~|~{1\leq j\leq m_i}\}$  forms a $K$-basis of $\Irr_\I^-(X, Y_i)$.

 Firstly, we show their linear independence. Suppose that
 $\Sigma_{j=1}^{m_i}k_{ij}\bar{u}_{ij}=0$ in $\Irr_\I^-(X, Y_i)$, where $k_{ij}\in K$. Without loss of generality, assume $k_{i1}\neq 0$. Then the morphism $(k_{i1}, \cdots, k_{im_i}):Y_i^{m_{i}}\to Y_i$ is split epic with an associated split mono $(k_{i1}^{-1}, 0, \cdots, 0)^t: Y_i\to Y_i^{m_{i}}$. Hence, it is from Corollary~\ref{cor:split irre} that the following composition
 $$\xymatrix{
 \Sigma_{j=1}^{m_i}k_{ij}{u}_{ij}=(X\ar[r]^{\quad u}
 &\coprod_{i=1}^{n}Y_i^{m_i}\ar[r]^{\quad p_i}
 &Y_i^{m_i}\ar[rr]^{(k_{i1}, \cdots, k_{im_i})}
 &
 &Y_i)
 }$$
 is left $\I$-irreducible. This contracts the hypothesis that it belongs to $\J\I(X, Y_i)$. Therefore, $k_{ij}=0,$ for all $j$, and the $\bar{u}_{ij}$ are linear independent.

   Next, we prove its generativity. Let $\bar{f}\in\Irr^-_\I(X, Y_i)$. Then there exists $g: Y\to Y_i$, written as $(g_{kl})^t_{1\leq k\leq n, 1\leq l\leq m_k}$, such that
   $$f=gu=\sum_{k=1}^{n}\sum_{l=1}^{m_k}g_{kl}u_{kl}.$$
   Note that $g_{il}$ is an endomorphism of $Y_i$ for $1\leq l\leq m_i$. Since $K$ is algebraically closed, then $\End(Y_i)/\J\simeq K$, and so
   $$g_{il}=\lambda_i{\id}_{Y_i}+ g'_{il},~\lambda_i\in K,~g'_{il}\in\J.$$
    On the other hand, if $k\neq i$, then $g_{kl}\in\J$. Because $u\in\I$, thus we have
    $$\bar{f}=\sum_{k=1}^{n}\sum_{l=1}^{m_k}\bar{g}_{kl}\bar{u}_{kl}=
    \sum_{l=1}^{m_i}\lambda_l\bar{u}_{il}.$$
Therefore the set $\{\bar{u}_{i1}, \bar{u}_{i2}, \cdots, \bar{u}_{im_i}\}$ is a basis of the $K$-space $\Irr^-_\I(X, Y_i)$. We finish the proof.
\end{proof}

Indeed, the converse of the above theorem is also true.

\begin{prop}\label{prop:detect source}
Let $\I$ be a functorially finite ideal.
Let $X$ and $Z$ be indecomposable and $Y=\coprod_{i=1}^{n}Y_i^{m_i}$ with $Y_i$ indecomposable and $Y_i\not\simeq Y_j$ for $i\neq j$.
\begin{enumerate}
  \item A morphism $f=(f_{ij})^t: X\to Y$ is an $\I$-source map provided
  \begin{itemize}
    \item [(i)] for each $i$, the set $\{\bar{f}_{i1}, \bar{f}_{i2}, \cdots, \bar{f}_{im_i}\}$ is a basis of $\Irr^-_\I(X, Y_i)$, and
    \item [(ii)] if $\Irr^-_\I(X, Y')\neq 0$ with $Y'$ indecomposable, then there is $i$ such that $Y'\simeq Y_i$.
  \end{itemize}
  \item A morphism $g=(g_{ij}):Y\to Z$ is an $\I$-sink map provided
  \begin{itemize}
    \item [(i)] for each $i$, the set $\{\bar{g}_{i1}, \bar{g}_{i2}, \cdots, \bar{g}_{im_i}\}$ is a basis of $\Irr^+_\I(Y_i, Z)$, and
    \item [(ii)] if $\Irr^+_\I(Y', Z)\neq 0$ with $Y'$ indecomposable, then there is $i$ such that $Y'\simeq Y_i$.
  \end{itemize}
\end{enumerate}
\end{prop}
\begin{proof}
  We only prove (1), the proof of (2) is dual.
  Since $\I$ is functorially finite, then there exists an $\I$-source map $s: X\to M$ of $X$. Because of the condition (ii), we have $M\simeq Y$. By Lemma~\ref{lem:irred}, $f\in\I$ and left $\I$-irreducible and then factors through $h$ by a morphism $u:M\to Y$. Since $h\in\I$, then $u$ is split epic. But $M\simeq Y$, so $u$ is isomorphic. Therefore $f$ is an $\I$-source.
\end{proof}

  When we choose $\I$ to be the Jacobson radical $\J$ of $\HT$, the above theorem coincides with a result of Happel \cite[Lemma 4.8]{happel1988}.

   With the above theorem, we are going to define the ideal mutation quiver. Let $\I$ be an Auslander-Reiten ideal of $\HT$,  the {\em valued $\I$-mutation quiver} $\Ga_\I(\HT)$ of $\HT$ is defined as
   \begin{itemize}
\item[$\bullet$] Vertices: the isomorphic classes of nonzero indecomposable objects.
\item[$\bullet$] Valued-arrows: for two indecomposable objects $X$ and $Y$ in $\HT$, there is an arrow from $X$ to $Y$ if $\dim_K\Irr^-_\I(X, Y)$ or $\dim_K\Irr^+_\I(X, Y)$ is nonzero, and in this case, the arrow has the value $(\dim_K\Irr^-_\I(X, Y),\dim_K\Irr^+_\I(X, Y))$. And delete all the valued arrows given by the trivial $\I$-mutation triangles, i.e.
    $$\xymatrix{
    X\ar@(ul, dl)_{(1, 1)}
    }, \forall~ \text{indecomposable}~X\in\Ob(\I).$$
\end{itemize}

When $\I=\J$, it is obvious that $\Ga_\J(\HT)$ is the classical valued Auslander-Reiten quiver.
Let $Z\in\ind\HT$, we define $\tau_\I(Z)=X$ if there exists an $\I$-mutation triangle of the form
$X\stackrel{u}\to Y\stackrel{v}\to Z\stackrel{w}\to X[1]$.
It is obvious that $\tau_\I:\ind\HT\setminus\Ob(\I)\to \ind\HT\setminus\Ob(\I)$ is a bijection.

\subsection{Factorization of morphisms}
In this subsection, $\I$ is assumed to be an Auslander-Reiten ideal of $\HT$.

\begin{prop}\label{prop: nil of an ideal}
Let $X, Y$  be indecomposable objects in $\HT$ and $n\geq 2$.  If $f\in \I^n(X, Y)$,   then
   \begin{enumerate}
     \item there exist $s \geq 1$, indecomposable modules $E_{1}, \ldots, E_{s}$ and morphisms $X \stackrel{h_{i}}{\longrightarrow} E_{i} \stackrel{g_{i}}{\longrightarrow} Y$ with $h_{i} \in \I$ and $g_{i}$ a sum of compositions of $n-1$ right $\I$-irreducible morphisms between indecomposables, and $f=\sum_{i=1}^{s} g_{i} h_{i}$.
     \item there exist $s \geq 1$, indecomposable modules $E_{1}, \ldots, E_{s}$ and morphisms $X \stackrel{h_{i}}{\longrightarrow} E_{i} \stackrel{g_{i}}{\longrightarrow} Y$ with $g_{i} \in \I$ and $h_{i}$ a sum of compositions of $n-1$ left $\I$-irreducible morphisms between indecomposables, and $f=\sum_{i=1}^{s} g_{i} h_{i}$.
   \end{enumerate}
\end{prop}
\begin{proof}
We only prove (1), (2) is proved dually. If $Y\in\Ob \I$, then $\id_Y$ is right $\I$-irreducible, hence $f=\id_Yf$ is a required decomposition. Assume $Y\notin\Ob \I$, we prove the assertion by induction on $n$.
When $n=2$, we can write $f=f_2f_1$ with $f_1\in\I(X, M)$ and $f_2\in\I(M, Y)$ (see the proof of Lemma~\ref{lem:irred} for reason). Consider the $\I$-mutation triangle of $Y$
$$\xymatrix{
  \tau_\I Y\ar[rr]^{(h_{1}, \cdots, h_{s})^t}&&\oplus_{i=1}^s E_{i} \ar[rr]^{\quad(g_{1}, \cdots, g_s)}&&Y\ar[rr]&& \tau_\I Y[1],
  }$$
then there are $\lambda_i: M\to E_i, i=1, \cdots, s$ such that $f_2= \sum_{i=1}^{s} g_i\lambda_i$. Let $h_i=\lambda_if_1$, then $f=\sum_{i=1}^{s} g_ih_i$ where $h_i\in\I$ and $g_i$ is right $\I$-irreducible between indecomposables.

Suppose the statement in (a) holds for $n-1$. If $f\in \I^n(X, Y)$, write $f=f_2f_1$ with $f_1\in\I$ and $f_2\in\I^{n-1}$. By the inductive assumption, $f_2=\sum_{i=1}^{l} a_ib_i$ with $b_{i} \in \I$ and $a_{i}$ a sum of compositions of $n-2$ right $\I$-irreducible morphisms between indecomposables. Note that each $b_if_1\in\I^2$, hence $b_if_1=\sum_{j=1}^{k} c_{ij}f_{ij}$ with $f_{ij}\in\I$ and $c_{ij}$ right $\I$-irreducible between indecomposables. Let $g_{ij}=a_ic_{ij}$ and $h_{ij}=f_{ij}$. Thus we have $f=\sum_{(i, j)}  g_{ij}h_{ij}$ with $h_{ij}\in I$ and $g_{ij}$ a sum of compositions of $n-1$ right $\I$-irreducible morphisms between indecomposables. We finish the proof.
\end{proof}

With the above proposition, we are going to prove a more general one.

\begin{prop}\label{prop: factor IJ}
   Let $X, Y$  be indecomposable objects in $\HT$ and $\R$ an ideal of $\HT$.
   \begin{enumerate}
     \item If $f\in \I^n\R(X, Y)$ for some $n\geq 1$.  Then there exist $s \geq 1$, indecomposable modules $E_{1}, \ldots, E_{s}$ and morphisms $X \stackrel{h_{i}}{\longrightarrow} E_{i} \stackrel{g_{i}}{\longrightarrow} Y$ with $h_{i} \in \R(X, E_i)$ and $g_{i}$ a sum of compositions of $n$ right $\I$-irreducible morphisms between indecomposables, and $f=\sum_{i=1}^{s} g_{i} h_{i}$.
     \item If $f\in \R\I^n(X, Y)$ for some $n\geq 1$.  Then there exist $s \geq 1$, indecomposable modules $E_{1}, \ldots, E_{s}$ and morphisms $X \stackrel{h_{i}}{\longrightarrow} E_{i} \stackrel{g_{i}}{\longrightarrow} Y$ with $g_{i} \in \R(X, E_i)$ and $h_{i}$ a sum of compositions of $n$ left $\I$-irreducible morphisms between indecomposables, and $f=\sum_{i=1}^{s} g_{i} h_{i}$.
   \end{enumerate}
\end{prop}
\begin{proof}
  We only prove (1), (2) can be proved dually. If $Y\in\Ob\I$, then $\id_Y$ is right $\I$-irreducible, hence $f=\id_Y^nf$ is a required decomposition. Assume $Y\notin\Ob \I$, when $n=1$,  write $f=f_2f_1$  with $f_1\in\R(X, M)$, $f_2\in\I(M, Y)$. Consider the $\I$-mutation triangle of $Y$
  $$\xymatrix{
  \tau_\I Y\ar[rr]^{(h_{1}, \cdots, h_{s})^t}&&\oplus_{i=1}^s E_{i} \ar[rr]^{\quad(g_{1}, \cdots, g_s)}&&Y\ar[rr]&& \tau_\I Y[1].
  }$$
 Then there are $\lambda_i: M\to E_i, i=1, \cdots, s$ such that $f_2= \sum_{i=1}^{s} g_i\lambda_i$. Let $h_i=\lambda_if_1$, then $f=\sum_{i=1}^{s} g_ih_i$ with $h_i\in\R$ and $g_i$ a right $\I$-irreducible between indecomposables.

Suppose the statement holds for $n-1$. If $f\in \I^n\R(X, Y)$, write $f=f_2f_1$ with $f_1\in\R$ and $f_2\in\I^n$. By Proposition~\ref{prop: nil of an ideal}, $f_2=\sum_{i=1}^{s} g_ih_i$ with $h_i\in\I$ and $g_i$ a  sum of compositions of $n-1$ right $\I$-irreducible morphisms between indecomposables. Note that $h_if_1\in\I\R$, then $h_if_1=\sum_{j=1}^{l}a_{ij}b_{ij}$  with $b_{ij}\in\R$ and $a_{ij}$ is right $\I$-irreducible between indecomposables. Hence $f=\sum_{i, j}g_ia_{ij}b_{ij}$ with $b_{ij}\in\R$ and $g_ia_{ij}$ a sum of compositions of $n$ right $\I$-irreducible morphisms between indecomposables.
\end{proof}

\section{Examples}\label{sec:examples}
In this section, we give examples to explain our results. Let's first recall some technical lemmas.

From the definition in \cite[Definition 1.4.1]{neeman2001triangulated}. A diagram
    $$\xymatrix{
    X\ar[r]^{f_1}\ar[d]_{f_2}
    &Y_1\ar[d]^{g_1}
    \\
    Y_2\ar[r]^{g_2}
    &Z
    }$$
    is called a {\em homotopy cartesian square} if  the corresponding sequence
    $$\xymatrix{
    X\ar[r]^{\binom{f_1}{f_2}\quad}&Y_1\oplus Y_2\ar[r]^{\quad(g_1, -g_2)}&Z\ar[r]&X[1]
    }$$
    is a triangle.

\begin{lem}
The following hold.
  \begin{enumerate}
    \item Each mesh square (with two intermediate terms) in an Auslander-Reiten quiver is a homotopy cartesian square.
    \item $($\cite[Proposition 6.11]{christensen2022good}$)$ The composition of two homotopy cartesian squares is again a homotopy cartesian square.
  \end{enumerate}
\end{lem}

The following lemma is used to find source or sink maps.

\begin{lem}
Let $\D$ be a subcategory of $\HT$ and $X$  an indecomposable object not in $\D$,
\begin{enumerate}
  \item Assume $f: X\to Y$  is a source map and $g: Y\to D$ is a left $\D$-approximation, then $gf$ is a left $\D$-approximation of $X$;
  \item Assume $u: Z\to X$ is a sink map and $v: D'\to Z$ is a left $\D$-approximation, then $uv$ is a left $\D$-approximation of $X$.
\end{enumerate}
\end{lem}

According to \cite{xiao2005locally}, let $\HT$ be a locally finite triangulated category, then it has Auslander-Reiten triangles, and if the Auslander-Reiten quiver has loops, it is of the form $\hat{L}_n$:
$$\xymatrix{
\ar@(ul,dl)X_1\ar@<0.25em>[r]
&X_2\ar@<0.25em>[l]\ar@<0.25em>[r]
&\cdots\ar@<0.25em>[l]\ar@<0.25em>[r]
&X_n\ar@<0.25em>[l]
}.$$

{ The following examples give Auslander-Reiten ideals which are not Jacobson radicals.}

\begin{exam}\label{exam:exsmple no tri}
Let $\HT$ be a locally finite triangulated $K$-category of type $\hat{L}_2$, that is $\HT$ has the Auslander-Reiten quiver
$$\xymatrix{
\ar@(ul,dl)_{\epsilon}X\ar@<0.25em>[rr]^{\al}
&
&Y\ar@<0.25em>[ll]^{\beta}
}$$
with relations $\al\beta=0$, $\epsilon^2+\beta\al=0$. $\HT$ has Auslander-Reiten triangles
\begin{align*}
   &\xymatrix{
X\ar^{\binom{\epsilon}{\al}\quad} [r]
&X\oplus Y\ar^{\quad(\epsilon, \beta)}[r]
& X\ar@{~>}[r]
&,
}\\
  &\xymatrix{
Y\ar^{\beta}[r]
&X\ar^{\al}[r]
&Y\ar@{~>}[r]
&.
}
\end{align*}

These two Auslander-Reiten triangles gives two homotopy cartesian squares
$$\xymatrix{
X\ar[r]^{\epsilon}\ar[d]_{\alpha}&X\ar[d]^{\epsilon}
\\
Y\ar[r]_{-\beta}&X
}
\quad\quad\quad\text{and}\quad\quad\quad
\xymatrix{
Y\ar[r]^{\beta}\ar[d]&X\ar[d]^{\alpha}
\\
0\ar[r]&Y.
}
$$

In $\HT$, $[1]=\id$ and the $\Hom$ space has $K$-basis
$$\{{\id}_X, {\id}_Y, \epsilon, \epsilon^2, \epsilon^3, \al, \al\epsilon, \beta, \epsilon\beta, \al\epsilon\beta\}.$$

(i)  Let $\I$ be the ideal generated by $\epsilon$. That is $\I =\langle\epsilon, \epsilon^2, \epsilon^3, \al\epsilon, \epsilon\beta, \al\epsilon\beta\rangle$. $\I$ is not the Jacobson radical of $\HT$. $\Gh_\I=\langle\epsilon^3, \al\epsilon, \epsilon\beta, \al\epsilon\beta\rangle=\CoGh_\I$. $\I$ is functorrially finite since $X\stackrel{\epsilon}\to X$ is both a right $\I$-approximation and  a left $\I$-approximation of $X$, $Y\stackrel{\epsilon\beta}\to X$ is a left $\I$-approximation  of $Y$ and $X\stackrel{\al\epsilon}\to Y$ is a right $\I$-approximation of $Y$. The morphisms
    $X\stackrel{\al\epsilon}\longrightarrow Y$ and $Y\stackrel{\epsilon\beta}\longrightarrow X$
   are both $\Gh_\I$-source maps and $\Gh_\I$-sinks maps, so $\I$ is an Auslander-Reiten ideal. Indeed,  the $\I$-mutation triangles are
\begin{align*}
   &  \xymatrix{
X\ar^{\epsilon}[r]
&X\ar^{\al\epsilon}[r]
&Y\ar@{~>}[r]
&,
}\\
  & \xymatrix{
Y\ar^{\epsilon\beta}[r]
&X\ar^{\epsilon}[r]
&X\ar@{~>}[r]
&.}
\end{align*}
These triangles are obtained by the pasting meshes,
$$\xymatrix@!0{
&X\ar[dr]^{\epsilon}&&
\\
X\ar[ur]^{\epsilon}\ar[dr]&&X\ar[dr]^{\alpha}&
\\
&Y\ar[ur]\ar@{-->}[dr]&&Y
\\
&&0\ar@{-->}[ur]&
}\quad \quad\text{and}\quad\quad
\xymatrix@!0{
&&X\ar[dr]^{\epsilon}&
\\
&X\ar[ur]^{\epsilon}\ar[dr]&&X
\\
Y\ar[ur]^{\beta}\ar@{-->}[dr]&&Y\ar[ur]&
\\
&0.\ar@{-->}[ur]&&
}$$

$\I/\J\I=\langle \bar{\epsilon}, \overline{\epsilon\beta}\rangle$, $\I/\I\J=\langle \bar{\epsilon}, \overline{\al\epsilon}\rangle$, and the valued $\I$-mutation quiver is
$$\xymatrix{
\ar@(ul,dl)_{(1, 1)}X\ar@<0.25em>[r]^{(0, 1)}
&Y.\ar@<0.25em>[l]^{(1, 0)}
}$$

(ii) Let $\I$ be { the} ideal generated by $\epsilon^2$, that is  { $\I=\langle \epsilon^2, \epsilon^3\rangle$}.
{  The ideal $\I$} is functorially finite but not an Auslander-Reiten ideal. Note that $\Gh_\I=\langle  \epsilon^2, \epsilon^3, \al, \al\epsilon, \beta, \epsilon\beta, \al\epsilon\beta\rangle$, $\al: X\to Y$ is an $\Gh_\I$-sink map of $Y$, but not a left $\Gh_\I$-approximation of $X$ since $\al\epsilon: X\to Y$ cannot factor through $\al$. Hence $\I$ isn't an Auslander-Reiten ideal.
\end{exam}

\begin{exam}
Let $\HT$ be a locally finite triangulated $K$-category of type $\hat{L}_3$, that is $\HT$ has the Auslander-Reiten quiver
$$\xymatrix{
\ar@(ul,dl)_{\epsilon}X\ar@<0.25em>[rr]^{\al}
&
&Y\ar@<0.25em>[ll]^{\beta}\ar@<0.25em>[rr]^{\gamma}
&
&Z\ar@<0.25em>[ll]^{\delta}
}$$
with relations $\epsilon^2+\beta\al=0$, $\al\beta+\delta\gamma=0$ and $\gamma\delta=0$. $\HT$ has Auslander-Reiten triangles
\begin{align*}
   & \xymatrix{
X\ar^{\binom{\epsilon}{\al}\quad} [r]
&X\oplus Y\ar^{\quad(\epsilon, \beta)}[r]
& X\ar@{~>}[r]
&
}, \\
   & \xymatrix{
Y\ar^{\binom{\beta}{\gamma}\quad}[r]
&X\oplus Z\ar^{\quad(\al, \delta)}[r]
&Y\ar@{~>}[r]
&
}, \\
    &\xymatrix{
Z\ar[r]^\delta
&Y\ar[r]^\gamma
&Z\ar@{~>}[r]
&
}.
\end{align*}
Let $\I$ be the ideal generated by $\epsilon$.
\begin{align*}
  \J = & \langle\epsilon, \epsilon^2, \epsilon^3, \epsilon^4, \epsilon^5, \al, \al\epsilon, \al\epsilon^2, \al\epsilon^3, \gamma\al, \gamma\al\epsilon, \beta, \epsilon\beta, \epsilon^2\beta, \epsilon^3\beta, \al\beta, \al\epsilon\beta, \al\epsilon^3\beta, \\
   &\ga, \gamma\al\epsilon\beta, \beta\delta, \epsilon\beta\delta, \al\epsilon\beta\delta, \gamma\al\epsilon\beta\delta \rangle,\\
  \I = & \langle\epsilon, \epsilon^2, \epsilon^3, \epsilon^4, \epsilon^5, \al\epsilon, \al\epsilon^2, \al\epsilon^3, \gamma\al\epsilon, \epsilon\beta, \epsilon^2\beta, \epsilon^3\beta, \al\epsilon\beta, \al\epsilon^3\beta, \\
   &\gamma\al\epsilon\beta, \epsilon\beta\delta, \al\epsilon\beta\delta, \gamma\al\epsilon\beta\delta \rangle,\\
   \Gh_\I=& \langle\epsilon^5, \al\epsilon^3, \gamma\al\epsilon, \epsilon^3\beta,  \al\epsilon\beta, \al\epsilon^3\beta, \gamma\al\epsilon\beta, \epsilon\beta\delta, \al\epsilon\beta\delta, \gamma\al\epsilon\beta\delta \rangle=\CoGh_\I.
\end{align*}
We remind the reader that $\al\epsilon^3\beta=\delta\gamma\delta\gamma=0$.
By an easy check, $\I$ is functorially finite. And since the morphisms $X\stackrel{\gamma\al\epsilon}\to Z$, $Y\stackrel{\al\epsilon\beta}\to Y$ and $Z\stackrel{\epsilon\beta\delta}\to X$ are both $\Gh_\I$-source maps and $\Gh_\I$-sink maps. Thus, $\I$ is an Auslander-Reiten ideal. Indeed, there are $\I$-mutation triangles:
\begin{align*}
   & \xymatrix{
X\ar^{\epsilon}[r]
&X\ar^{\gamma\al\epsilon}[r]
&Z\ar@{~>}[r]
&
}, \\
   &  \xymatrix{
Y\ar^{\epsilon\beta}[r]
&X\ar^{\al\epsilon}[r]
&Y\ar@{~>}[r]
&
},\\
   & \xymatrix{
Z\ar[r]^{\epsilon\beta\delta}
&X\ar[r]^{\epsilon}
&X\ar@{~>}[r]
&
}.
\end{align*}
These triangles are obtained by the pasting meshes,
$$
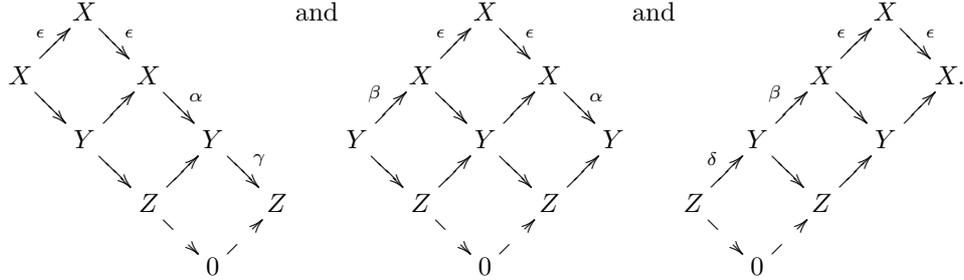
\xymatrix@!0{
&X\ar[dr]^{\epsilon}&&&
\\
X\ar[ur]^{\epsilon}\ar[dr]&&X\ar[dr]^{\alpha}&&
\\
&Y\ar[ur]\ar[dr]&&Y\ar[dr]^{\gamma}&
\\
&&Z\ar[ur]\ar@{-->}[dr]&&Z
\\
&&&0\ar@{-->}[ur]&
}\text{and}
\xymatrix@!0{
&&X\ar[dr]^{\epsilon}&&
\\
&X\ar[ur]^{\epsilon}\ar[dr]&&X\ar[dr]^{\alpha}&
\\
Y\ar[ur]^{\beta}\ar[dr]&&Y\ar[ur]\ar[dr]&&Y
\\
&Z\ar@{-->}[dr]\ar[ur]&&Z\ar[ur]&
\\
&&0\ar@{-->}[ur]&&
}\text{and}
\xymatrix@!0{
&&&X\ar[dr]^{\epsilon}&
\\
&&X\ar[ur]^{\epsilon}\ar[dr]&&X.
\\
&Y\ar[ur]^{\beta}\ar[dr]&&Y\ar[ur]&
\\
Z\ar@{-->}[dr]\ar[ur]^{\delta}&&Z\ar[ur]&&
\\
&0\ar@{-->}[ur]&&&
}$$
\end{exam}

In the following examples, we draw ideal mutation quivers of triangulated categories concerning a functorially finite $\tau$-stable subcategory.

\begin{exam}
  As in the above example, $\D^b(A\smod)$  has Auslander-Reiten triangles. Indeed, $\D^b(A\smod)$ has the Auslander-Reiten quiver
  $$\xymatrix@!0{
  &&{P_1}\ar[dr]
  &&{P_3[1]}\ar@{-->}[ll]_\tau\ar[dr]
  && {S_2[1]}\ar@{-->}[ll]_\tau\ar[dr]
  &&{S_1[1]}\ar[dr]\ar@{-->}[ll]_\tau
  &&{P_1[2]}\ar@{-->}[ll]_\tau
  \\
  \dots
  &{P_2}\ar[ur]\ar[dr]
  &&{I_2}\ar[ur]\ar[dr]\ar@{-->}[ll]_\tau
  &&{P_2[1]}\ar[ur]\ar[dr]\ar@{-->}[ll]_\tau
  &&{I_2[1]}\ar[dr]\ar[ur]\ar@{-->}[ll]_\tau
  &&{P_2[2]}\ar[ur]\ar@{-->}[ll]_\tau
  &\cdots
  \\
  P_3\ar[ur]
  &&S_2\ar[ur]\ar@{-->}[ll]_\tau
  &&S_1\ar[ur]\ar@{-->}[ll]_\tau
  &&P_1[1]\ar[ur]\ar@{-->}[ll]_\tau
  && P_3[2]\ar@{-->}[ll]_\tau\ar[ur]
  &&.
  }$$
  Set $\D=\{\tau^i P_3~|~i\in\mathbb{Z}\}$, by Corollary~\ref{cor: object mutation}, $[\D]$ is an Auslander-Reiten ideal. The valued $[\D]$-mutation quiver is gluing by the following three quivers, note that the first two  have a common item $P_1[1]$, the latter two have a common item $S_1$ and the first, and the third ones have a common item $P_3[2]$.
  $$\xymatrix@!0{
  &P_2\ar[dr]
  && P_3[1]\ar[dr]\ar@{-->}[ll]_{\tau_\D}
  &&I_2[1]\ar[dr]\ar@{-->}[ll]_{\tau_\D}
  &&P_1[2]\ar@{-->}[ll]_{\tau_\D}\ar[dr]
  &
  \\
  \cdots\ar[ur]
  &&S_2\ar[ur]
  &&P_1[1]\ar[ur]
  &&P_3[2]\ar[ur]
  &&\cdots~,
  }$$
    $$\xymatrix@!0{
  &P_1\ar[dr]
  && P_2[1]\ar[dr]\ar@{-->}[ll]_{\tau_\D}
  &&S_1[1]\ar[dr]\ar@{-->}[ll]_{\tau_\D}
  &
  \\
  \cdots\ar[ur]
  &&S_1\ar[ur]
  &&P_1[1]\ar[ur]
  &&\cdots~,
  }$$
  $$\xymatrix@!0{
  &I_2\ar[dr]
  && S_2[1]\ar[dr]\ar@{-->}[ll]_{\tau_\D}
  &&P_2[2]\ar[dr]\ar@{-->}[ll]_{\tau_\D}
  &
  \\
  \cdots\ar[ur]
  &&S_1\ar[ur]
  &&P_3[2]\ar[ur]
  &&\cdots~.
  }$$
\end{exam}

In the following, we draw an ideal mutation quiver of a cluster category. For a wealth of information on the theory of cluster categories, we refer to \cite{buan2006tilting}.

\begin{exam}
Let $Q$ be the quiver
 $$1\longrightarrow 2\longrightarrow 3\longrightarrow 4.$$
 Its corresponding cluster category $\C_Q$  has the Auslander-Reiten quiver
 $$\xymatrix@!0{
  &&&P_1\ar[dr]&&P_4[1]\ar@{-->}[ll]_\tau\ar[dr]
  && P_4\ar@{-->}[ll]_\tau\ar[dr]&&
  \\
  &&P_2\ar[ur]\ar[dr]&&I_3\ar[ur]\ar[dr]\ar@{-->}[ll]_\tau
  &&P_3[1]\ar[ur]\ar[dr]\ar@{-->}[ll]_\tau&& P_3\ar[dr]\ar@{-->}[ll]_\tau&
  \\
  &P_3\ar[ur]\ar[dr]&&\mathop{}^{2}_{3}\ar[ur]\ar[dr]\ar@{-->}[ll]_\tau
  &&I_2\ar[ur]\ar[dr]\ar@{-->}[ll]_\tau
  &&P_2[1]\ar[ur]\ar[dr]\ar@{-->}[ll]_\tau&& P_2\ar[dr]\ar@{-->}[ll]_\tau
  \\
  P_4\ar[ur]&&S_3\ar[ur]\ar@{-->}[ll]_\tau&&S_2\ar[ur]\ar@{-->}[ll]_\tau
  &&S_1\ar[ur]\ar@{-->}[ll]_\tau&&P_1[1]\ar[ur]\ar@{-->}[ll]_\tau&&P_1\ar@{-->}[ll]_\tau
  }$$
  We set $\D=\{P_1, P_4[1], P_4, S_3, S_2, S_1, P_1[1]\}$, $\I=[\D]$, by Corollary~\ref{cor: object mutation}, $\I$ is an Auslander-Reiten ideal. The valued $\I$-mutation quiver is
   $$\xymatrix@!0{
  &\quad P_1\quad\ar[dr]^{01}
  &&P_4[1]\ar[dr]^{01}
  &&S_3\ar[dr]^{01}
  &&P_4[1]\ar[dr]^{01}
  &&P_4\ar[dr]^{01}
  &&P_1\ar[dr]^{01}
  &&P_4\ar[dr]^{01}
  &
  \\
  P_3\ar[ur]^{10}\ar[dr]_{10}
  &&I_3\ar[ur]^{10}\ar[dr]_{10}\ar@{-->}[ll]_{\tau_\I}
  &&P_2[1]\ar[ur]^{10}\ar[dr]_{10}\ar@{-->}[ll]_{\tau_\I}
  &&\mathop{}^{2}_{3}\ar[ur]^{10}\ar[dr]_{10}\ar@{-->}[ll]_{\tau_\I}
  &&P_3[1]\ar[ur]^{10}\ar[dr]_{10}\ar@{-->}[ll]_{\tau_\I}
  &&P_2\ar[ur]^{10}\ar[dr]_{10}\ar@{-->}[ll]_{\tau_\I}
  &&I_2\ar[ur]^{10}\ar[dr]_{10}\ar@{-->}[ll]_{\tau_\I}
  &&P_3\ar@{-->}[ll]_{\tau_\I}
  \\
   &S_3\ar[ur]_{01}
   &&S_1\ar[ur]_{01}
   &&P_1[1]\ar[ur]_{01}
   &&S_2\ar[ur]_{01}
   &&P_1[1]\ar[ur]_{01}
   &&S_2\ar[ur]_{01}
   &&S_1\ar[ur]_{01}
   &
  }$$
  where the notation $01$ denotes $(0, 1)$ and $10$ denotes $(1, 0)$.
  \end{exam}

Let $\HT$ be a $\Hom$-finite Krull-Schmidt triangulated $K$-category having Auslander-Reiten triangles and $\D$ a functorially finite $\tau$-stable subcategory. Roughly speaking, the Auslander-Reiten quiver $\Gamma({\HT})$ contains all information of morphisms in $\HT$. And the $[\D]$-mutation quiver $\Gamma_{\D}(\HT)$ contains all information of morphisms in $[\D]$. And the Auslander-Reiten quiver $\Gamma({\HT/[\D]})$ of the quotient category $\HT[\D]$, a subquiver of $\Gamma({\C_Q})$ by deleting the vertices corresponding to objects of $\D$ along with the arrows into or out of such vertices (see \cite[Theorem 4.2(ii)]{jorgensen2010quotients}), contains all information of morphisms that are not in $[\D]$. We ask the following question.\\

  {\bf Question:} Can we recover $\Gamma({\HT})$ via $\Gamma_{\D}(\HT)$ and $\Gamma(\HT/[\D])$?\\

{\bf Acknowledgments:}  The authors would like to thank Professor Panyue Zhou and the anonymous reviewer for their careful reading and useful suggestions to improve the article.
%%%%%%%%%%%%%%%
\bibliographystyle{plain}
%\bibliography{MutationV9}

\end{document}